\documentclass[11pt,a4paper,reqno]{amsart}

\usepackage{iftex}

\ifPDFTeX
  \usepackage[T1]{fontenc}
  \usepackage[utf8]{inputenc}
  \usepackage{lmodern}
\else
  \usepackage{fontspec}
  \usepackage{lmodern}
  \defaultfontfeatures{Ligatures=TeX}
\fi

\usepackage[english]{babel}
\usepackage{csquotes}
\usepackage{microtype}

\usepackage[
  a4paper,
  top=1.4in,
  bottom=1.8in,
  left=1.0in,
  right=1.0in
]{geometry}
\usepackage{setspace}
\usepackage{ragged2e}

\usepackage{amsmath,amsthm,amssymb}
\usepackage{mathtools}
\usepackage{bm}
\usepackage{mathrsfs}
\usepackage{esint}
\usepackage{centernot}
\usepackage{cancel}
\usepackage{tensor}

\usepackage{graphicx}
\usepackage{float}
\usepackage{placeins}
\usepackage[font=small,labelfont=bf]{caption}
\usepackage{subcaption}
\usepackage{wrapfig}
\usepackage{rotating}
\usepackage{adjustbox}
\usepackage{xcolor}

\usepackage{tikz}
\usetikzlibrary{
  arrows.meta,
  calc,
  decorations.pathreplacing,
  matrix,
  positioning
}
\usepackage{tikz-cd}

\usepackage{booktabs}
\usepackage{array}
\usepackage{tabularx}
\usepackage{multirow}
\usepackage{longtable}
\usepackage{makecell}
\usepackage{threeparttable}

\usepackage{enumitem}
\usepackage{quoting}
\usepackage{comment}
\usepackage{verbatim}
\usepackage{fancyvrb}
\usepackage{listings}
\usepackage{etoolbox}
\usepackage{xparse}
\usepackage{calc}

\usepackage[
  backend=biber,
  style=numeric,
  sorting=nyt,
  sortcites=true,
  giveninits=true,
  maxbibnames=99,
  url=false,
  doi=false,
  isbn=false,
  backref=false
]{biblatex}
\DeclareBibliographyCategory{bibliographyentriesused}
\newtoggle{bibliographycheckactive}
\newtoggle{bibliographycheckallincluded}

\AtEveryCitekey{%
  \addtocategory{bibliographyentriesused}{\thefield{entrykey}}%
}

\newcommand{\markbibliographyentryused}[1]{%
  \addtocategory{bibliographyentriesused}{#1}%
}
\let\originalnociteforbibliographycheck\nocite
\RenewDocumentCommand{\nocite}{m}{%
  \ifstrequal{#1}{*}
    {\global\toggletrue{bibliographycheckallincluded}}
    {\forcsvlist{\markbibliographyentryused}{#1}}%
  \originalnociteforbibliographycheck{#1}%
}

\AtEveryBibitem{%
  \iftoggle{bibliographycheckactive}{%
    \ifkeyword{bibcheck-ignore}
      {}
      {\PackageWarningNoLine{bibliography-check}{%
         Uncited bibliography entry `\thefield{entrykey}'%
       }}%
  }{}%
}

\makeatletter
\AtEndDocument{%
  \iftoggle{bibliographycheckallincluded}
    {}
    {%
      \ifcsname norefnames\endcsname\norefnames\fi
      \ifcsname setoffmsgs\endcsname\setoffmsgs\fi
      \begin{refsection}
        \originalnociteforbibliographycheck{*}%
        \begingroup
          \toggletrue{bibliographycheckactive}%
          \let\blx@warn@bibempty\relax
          \setbox0=\vbox{%
            \hsize=\textwidth
            \printbibliography[
              notcategory=bibliographyentriesused,
              heading=none
            ]%
          }%
        \endgroup
      \end{refsection}
    }%
}
\makeatother

\usepackage[
  unicode=true,
  colorlinks=true,
  linkcolor=blue,
  citecolor=blue,
  urlcolor=black
]{hyperref}
\usepackage{bookmark}
\usepackage[
  nameinlink,
  noabbrev,
  capitalise
]{cleveref}

\makeatletter
\AtBeginDocument{%
  \let\refcheckoriginalcref\cref
  \RenewDocumentCommand{\cref}{s m}{%
    \forcsvlist{\wrtusdrf}{#2}%
    \IfBooleanTF{#1}
      {\refcheckoriginalcref*{#2}}%
      {\refcheckoriginalcref{#2}}%
  }%
  \let\refcheckoriginalCref\Cref
  \RenewDocumentCommand{\Cref}{s m}{%
    \forcsvlist{\wrtusdrf}{#2}%
    \IfBooleanTF{#1}
      {\refcheckoriginalCref*{#2}}%
      {\refcheckoriginalCref{#2}}%
  }%
  \let\refcheckoriginalcrefrange\crefrange
  \RenewDocumentCommand{\crefrange}{s m m}{%
    \wrtusdrf{#2}\wrtusdrf{#3}%
    \IfBooleanTF{#1}
      {\refcheckoriginalcrefrange*{#2}{#3}}%
      {\refcheckoriginalcrefrange{#2}{#3}}%
  }%
  \let\refcheckoriginalCrefrange\Crefrange
  \RenewDocumentCommand{\Crefrange}{s m m}{%
    \wrtusdrf{#2}\wrtusdrf{#3}%
    \IfBooleanTF{#1}
      {\refcheckoriginalCrefrange*{#2}{#3}}%
      {\refcheckoriginalCrefrange{#2}{#3}}%
  }%
}
\makeatother

\theoremstyle{plain}
\newtheorem{theorem}{Theorem}[section]
\newtheorem{lemma}[theorem]{Lemma}
\newtheorem{proposition}[theorem]{Proposition}
\newtheorem{corollary}[theorem]{Corollary}

\theoremstyle{definition}
\newtheorem{definition}[theorem]{Definition}

\theoremstyle{remark}
\newtheorem{remark}[theorem]{Remark}

\crefname{theorem}{theorem}{theorems}
\crefname{lemma}{lemma}{lemmas}
\crefname{proposition}{proposition}{propositions}
\crefname{corollary}{corollary}{corollaries}
\crefname{conjecture}{conjecture}{conjectures}
\crefname{definition}{definition}{definitions}
\crefname{assumption}{assumption}{assumptions}
\crefname{condition}{condition}{conditions}
\crefname{example}{example}{examples}
\crefname{problem}{problem}{problems}
\crefname{question}{question}{questions}
\crefname{remark}{remark}{remarks}
\crefname{notation}{notation}{notations}
\crefname{observation}{observation}{observations}

\numberwithin{equation}{section}
\allowdisplaybreaks[3]
\patchcmd{\abstract}{\scshape\abstractname}{\bfseries\abstractname}{}{}

\begin{document}

\title[Curvature Estimates]{Global Curvature Estimates for \\ 
\(\sigma_k\) Curvature Equations with \(k\geq n/2\)}

 \author{Jin Yan}
\address{INSTITUTE OF MATHEMATICS, ACADEMY OF MATHEMATICSAND SYSTEMS SCIENCE, CHINESE ACADEMY OF SCIENCES, BEIJING, 100190, CHINA}
\email{yanjin@amss.ac.cn}

\subjclass[2020]{Primary 35J60; Secondary 35B45, 53C42}
\keywords{fully nonlinear elliptic equations,
\(\sigma_k\)-curvature equations, global curvature estimates,
Hessian equations, concavity inequalities, G{\aa}rding roots}

\begin{abstract}
We establish a new concavity inequality for the elementary symmetric
function \(\sigma_k\), which controls the spectral quadratic form
arising from the second variation of
\(\log\lambda_{\max}\). The proof is based on an easy--hard
decomposition of the spectral variables. The easy region is treated
using an optimal constrained concavity estimate for \(\sigma_k\),
whereas the hard region is analyzed through G{\aa}rding-root
coordinates and the concavity of the ordered partial sums of the inverse roots.  As an application, for \(n/2\leq k<n\), we obtain global curvature estimates for closed star-shaped \(k\)-convex hypersurfaces satisfying \(\sigma_k(\kappa)=f(X,\nu)\) with a general positive right-hand side, together with the corresponding global-to-boundary estimates for Euclidean Hessian equations.
\end{abstract}

\maketitle
\tableofcontents

\section{Introduction}
Let
\(M\subset\mathbb R^{n+1}\) be a smooth closed hypersurface, let
\(\nu\) be its unit outer normal, and let
\(\kappa=(\kappa_1,\ldots,\kappa_n)\) denote its principal curvature
vector. We consider
\begin{eqnarray}\label{eq:intro-prescribed-curvature}
    \sigma_k(\kappa(X))
    =
    f(X,\nu(X))>0,
    \qquad X\in M,
\end{eqnarray}
where \(\sigma_k\) is the \(k\)-th elementary symmetric function. The natural ellipticity condition for
\eqref{eq:intro-prescribed-curvature} is
\begin{eqnarray*}
    \kappa(X)\in\Gamma_k,
    \qquad
    \Gamma_k
    :=
    \left\{
        \lambda\in\mathbb R^n:
        \sigma_j(\lambda)>0,\quad 1\leq j\leq k
    \right\}.
\end{eqnarray*}
Throughout this paper, a hypersurface satisfying this condition is
called \(k\)-convex.

Many classical geometric problems arise from
\eqref{eq:intro-prescribed-curvature} by prescribing data with a
particular dependence on the geometric variables. When
\(f=f(\nu)\), the most famous example is the Minkowski problem, which
prescribes the Gauss curvature as a function of the unit outer normal,
or equivalently the surface area measure in Gauss-map coordinates.
Its existence and regularity theory was developed in the classical
works of Minkowski \cite{Minkowski1903}, Nirenberg
\cite{Nirenberg1953}, Cheng--Yau \cite{ChengYau1976}, and Pogorelov
\cite{Pogorelov1978}. More general inverse-Gauss-map prescriptions of
Weingarten curvature were studied by Guan--Guan
\cite{GuanGuan2002} and Sheng--Trudinger--Wang
\cite{ShengTrudingerWang2004}. Closely related are the
Christoffel--Minkowski problems, where elementary symmetric
functions of the principal curvature radii are prescribed on the
sphere. Guan--Ma \cite{GuanMa2003} developed the spherical-Hessian
and full-rank framework for these problems, and their
Weingarten-curvature and admissible-solution aspects were further
investigated by Guan--Lin--Ma \cite{GuanLinMa2006} and
Guan--Ma--Zhou \cite{GuanMaZhou2006}. When \(f=f(X)\), prescribed mean-curvature equations for star-shaped
hypersurfaces were studied by Bakelman--Kantor
\cite{BakelmanKantor1974} and Treibergs--Wei
\cite{TreibergsWei1983}, while the prescribed Gauss-curvature problem
was treated by Oliker \cite{Oliker1984}. The
corresponding fully nonlinear Weingarten equations were treated by
Caffarelli--Nirenberg--Spruck \cite{CNS1986}, with related extensions
to space forms due to Li--Oliker \cite{LiOliker2002} and
Barbosa--Lira--Oliker \cite{BarbosaLiraOliker2002}.

The Alexandrov and curvature measure problems have a different
geometric form: the prescribed datum is assigned through the radial
map rather than the Gauss map. The case
corresponding to the classical Alexandrov problem was studied by
Alexandrov \cite{Alexandrov1942}, while the intermediate
curvature-measure problems were treated by Guan--Lin--Ma
\cite{GuanLinMa2009} and Guan--Li--Li \cite{GuanLiLi2012}. In the
smooth setting, the equation for the prescribed \((n-k)\)-th
curvature measure takes the form
\begin{eqnarray*}
    \sigma_k(\kappa(X))
    =
    \frac{\langle X,\nu\rangle}{|X|^{n+1}}
    \varphi\left(\frac{X}{|X|}\right).
\end{eqnarray*}

Taken together, these works of Pengfei Guan and his collaborators
connect the classical Gauss-map and Christoffel--Minkowski
formulations with curvature equations involving both the position and
normal variables. In the curvature-measure equation above, however,
the dependence on \(X\) and \(\nu\) is constrained by its geometric
origin. After establishing
the required curvature estimate for this specially structured
equation, Guan--Li--Li \cite[Remark~3.5]{GuanLiLi2012} asked whether
the same type of estimate remains valid for a general positive
function \(f(X,\nu)\). More precisely, they asked whether an a priori
\(C^1\) bound for a smooth compact \(k\)-admissible hypersurface
satisfying
\begin{eqnarray*}
    \sigma_k(\kappa(X))
    =
    f(X,\nu(X)),
    \qquad 1<k<n,
\end{eqnarray*}
implies a global \(C^2\) estimate controlled only by the \(C^1\)
geometry of the hypersurface and the prescribed data. In the same
remark, they posed the analogous global-to-boundary problem for
admissible solutions of
\begin{eqnarray*}
    \sigma_k(D^2u)
    =
    f(x,u,Du).
\end{eqnarray*}
These questions are the direct point of departure for the present
paper.

The dependence on the normal variable is the main source of
difficulty. Indeed, writing \(F=\sigma_k\) and
\(F^{ii}=\partial F/\partial\kappa_i\), at a point where the second
fundamental form is diagonal, the first derivative of
\eqref{eq:intro-prescribed-curvature} is
\begin{eqnarray*}
    \sum_{i=1}^nF^{ii}h_{ii;l}
    =
    d_Xf(e_l)+\kappa_l\,d_\nu f(e_l).
\end{eqnarray*}
When \(f\) is independent of \(\nu\), the foundational estimates of
Caffarelli, Nirenberg, and Spruck
\cite{CNS1984,CNS1985,CNS1986,CNS1988} and their subsequent
extensions apply to a broad class of curvature equations. Related second-order estimates were obtained by Ivochkina
\cite{Ivochkina1990,Ivochkina1991}, Trudinger
\cite{Trudinger1990}, Bo Guan \cite{Guan1999,Guan2014}, and
Sheng--Urbas--Wang \cite{ShengUrbasWang2004}. Extensions to
warped-product settings were obtained by Chen--Li--Wang
\cite{ChenLiWang2018}, while related Hessian-type equations were
considered by Li--Ren--Wang \cite{LiRenWang2019}. For a
general \(f(X,\nu)\), however, a second differentiation introduces
terms in which derivatives of the second fundamental form interact
with the normal derivatives of \(f\). At a maximum point of the
largest-principal-curvature test function, the resulting third-order
terms are not controlled by the standard concavity of
\(\sigma_k^{1/k}\) alone. A geometrically structured example of such normal dependence arises
from the Gauss equation for immersed hypersurfaces in Riemannian
manifolds. In this setting, Guan--Lu \cite{GuanLu2017} obtained
curvature estimates under the assumption of nonnegative extrinsic
scalar curvature and applied them to Weyl-type isometric embedding
problems.

One way to overcome this difficulty is to impose additional
convexity on $\kappa$. Guan--Ren--Wang \cite{GuanRenWang2015} obtained global estimates for convex solutions with general \(f(X,\nu)\). Estimates for
semiconvex hypersurfaces in hyperbolic space were later obtained by Lu
\cite{Lu2023Semiconvex}. These assumptions can control the third-order terms generated by general $f(X,\nu)$, but cannot derive the existence result directly.

Without such additional convexity, Guan--Ren--Wang
\cite{GuanRenWang2015} also derived the result when \(k=2\);
see also Spruck--Xiao \cite{SpruckXiao2017} for a new proof. For \(k\geq3\), however, the problem becomes more
involved. Ren--Wang \cite{RenWang2019} established the global curvature
estimate for \(k=n-1\). Later, in \cite{RenWang2023}, they
showed that throughout the supercritical range \(k>n/2\), the same
problem can be reduced to a quadratic-form inequality in the spectral
variables, and verified that inequality for \(k=n-2\).
Very recently, Lu--Tsai \cite{LuTsai2026} gave a new proof for $k=n-1$ based on the previously established semiconvexity result.

A key feature of the Ren--Wang argument is the use of an auxiliary
function with the symmetric leading term
\begin{eqnarray*}
    \log\log P,\qquad P=\sum_{i=1}^n e^{\kappa_i}.
\end{eqnarray*}
Such an auxiliary function appears to provide stronger convexity, but it also makes the quadratic form more complicated. Numerical experiments suggest that their conjectured inequality is valid, but a proof is not yet available. In the present paper, we instead work directly with the simpler leading term
\begin{eqnarray*}
    \log\kappa_{\max},
\end{eqnarray*}
which makes the quadratic form $\mathcal{Q}_{\gamma}$ cleaner.

Let
\(\lambda\in\Gamma_k\), with
\(a:=\lambda_1>\lambda_2\geq\cdots\geq\lambda_n\), and set
\begin{eqnarray*}
    F:=\sigma_k(\lambda),
    \qquad
    F^{ii}:=\sigma_{k-1}(\lambda|i),\qquad F^{ii,jj}
    :=
    \begin{cases}
        \sigma_{k-2}(\lambda|ij),& i\neq j,\\
        0,& i=j.
    \end{cases}
\end{eqnarray*}
For
\(\xi=(\xi_1,\ldots,\xi_n)\in\mathbb R^n\), define
\(\mathcal Q_\gamma(\lambda;\xi)\) by
\begin{eqnarray}\label{Qgamma}
    \mathcal Q_{\gamma}(\lambda;\xi)
    :=
    \frac{2}{aF}
    \left(\sum_{i=1}^nF^{ii}\xi_i\right)^2
    -\gamma\frac{F^{11}}{a^2}\xi_1^2-\frac1a\sum_{i,j=1}^nF^{ii,jj}\xi_i\xi_j
    +\frac2a\sum_{p=2}^n
    \frac{F^{pp}}{a-\lambda_p}\,\xi_p^2.
\end{eqnarray}

Our crucial concavity inequality is:
\begin{theorem}\label{thm-crucial-ineq}
Let \(n\geq3\), \(2\leq k<n\), and
\(\lambda\in\Gamma_k\), with
\(a=\lambda_1>\lambda_2\geq\cdots\geq\lambda_n\). Assume that 
$$
0<\gamma<\min\{\frac{2k}{n}, 1+\frac{2k-n}{2k^2+n}\}.
$$
Then there exists a small constant $ \eta_*=
    \eta_*(n,k,\gamma)>0$ such that whenever
\[
   0<\frac{F}{a^k}\leq\eta_*,
\]
we have
\begin{eqnarray}\label{Qbetagamma>0}
    \mathcal Q_{\gamma}(\lambda;\xi)\ge0,
    \qquad
    \forall\xi\in\mathbb R^n.
\end{eqnarray}
\end{theorem}
As Corollary~\ref{cor:jacobi} shows, when \(k>n/2\), our concavity
inequality is almost equivalent to the Jacobi inequality for
\(b=\log\lambda_{\max}(D^2u)\). Since Jacobi inequalities play a
central role in interior \(C^2\) estimates, we briefly recall the
relevant results for the \(k\)-Hessian equation
\begin{eqnarray}\label{eq:intro-k-hessian}
    \sigma_k(D^2u)=f.
\end{eqnarray}

Heinz \cite{Heinz1959} first established the interior \(C^2\) estimate
when \(n=k=2\), namely, for the two-dimensional Monge--Amp\`ere
equation. Pogorelov \cite{Pogorelov1978} constructed singular solutions
to the Monge--Amp\`ere equation in dimensions \(n\geq3\), and Urbas
\cite{Urbas1990} extended this construction to the
\(\sigma_k\)-equations for \(n\geq k\geq3\). These examples show that
interior \(C^2\) estimates cannot hold in general for \(k\geq3\)
without additional structural assumptions, leaving the
\(2\)-Hessian equation as the only unresolved case.

A major breakthrough was made by Warren--Yuan
\cite{WarrenYuan2009}, who used the special Lagrangian structure to
establish the interior \(C^2\) estimate for
\(\sigma_2(D^2u)=1\) in dimension three. It is worth noting that they
first introduced a Jacobi
inequality for
\(b_1=\log\sqrt{1+\lambda_1^2}\). For general positive
right-hand sides in dimension three, Qiu \cite{Qiu2024Hessian}
first derived a trace Jacobi inequality and introduced a
doubling argument, which led to the result. More recently, Shankar--Yuan \cite{ShankarYuan2025} proved the
constant-right-hand-side case in dimension four by combining Qiu's
doubling argument with ideas from Chaudhuri--Trudinger
\cite{ChaudhuriTrudinger2005} and Savin's small-perturbation theorem
\cite{Savin2007}. Fan \cite{Fan2026} subsequently extended the
dimension-four result to positive \(C^{1,1}\) variable right-hand
sides by following this approach and using a generalization of Savin's
theorem. For \(n\geq5\), the problem without additional assumptions remains
open. Guan--Qiu \cite{GuanQiu2019} obtained estimates under the
additional condition
\(\sigma_3(D^2u)\geq-A\), while
Shankar--Yuan \cite{ShankarYuan2020} treated semiconvex solutions of
\(\sigma_2(D^2u)=1\). For related interior estimates for curvature
equations and other geometric fully nonlinear equations, we refer to
Guan--Qiu \cite{GuanQiu2019}, Qiu \cite{Qiu2024Curvature},
Qiu--Zhou \cite{QiuZhou2024}, and the references therein.

\

The principal geometric application of the preceding inequality gives an affirmative answer to the curvature-estimate question of Guan--Li--Li for \(k\geq n/2\). The result is new for \(n/2\leq k\leq n-3\); for the analogous Hessian problem, see Theorem~\ref{thm:euclidean-hessian-application}.
\begin{theorem}\label{thm:curvature-application}
Let \(n\geq3\) and \(n/2\leq k<n\). Suppose that
\(M\subset\mathbb R^{n+1}\) is a closed smooth star-shaped \(k\)-convex hypersurface
satisfying
\begin{eqnarray}\label{eq:prescribed-curvature-application}
    \sigma_k(\kappa(X))
    =
    f(X,\nu(X)),
    \qquad X\in M,
\end{eqnarray}
for some positive function \(f(X,\nu)\in C^2(\Gamma)\), where
\(\Gamma\) is an open neighborhood of the unit normal bundle of \(M\)
in \(\mathbb R^{n+1}\times\mathbb S^n\). Then there exists a constant
\(C\), depending only on $n$, $k$, $\|M\|_{C^1}$, $\inf_\Gamma f$, $\|f\|_{C^2(\Gamma)}$,
such that
\begin{eqnarray}\label{eq:global-curvature-application}
    \max_{\substack{X\in M\\1\leq i\leq n}}
    \kappa_i(X)
    \leq C.
\end{eqnarray}
\end{theorem}

\medskip
\noindent

We now briefly outline the proof of
Theorem~\ref{thm-crucial-ineq}. By homogeneity, we
first normalize \(a=\lambda_1=1\) and divide the spectral variables
into two regions according to the size of \(F^{11}/F\). In the easy
region, an exact constrained concavity estimate for \(\sigma_k\),
together with the smallness of \(F\), controls the unfavorable
\(\xi_1^2\)-term directly. In the hard region, we represent the
\(\sigma_k\)-level set as a graph over the remaining spectral
variables and rewrite this graph in terms of its G{\aa}rding roots.
After passing to the inverse-root variables, the second-order chain
rule decomposes \(\mathcal Q_\gamma\) into an explicit root-variable
quadratic form and a remainder. Completing the square and applying an
explicit matrix estimate give the positivity of the former when \(F\)
is sufficiently small, while the concavity of the ordered partial sums
of the inverse G{\aa}rding roots, together with summation by parts,
shows that the latter is nonnegative. Combining the two regions and
restoring the original scale proves Theorem~\ref{thm-crucial-ineq}.

\medskip
\noindent

The remainder of the paper is organized as follows. Section~\ref{notation} introduces the basic notation and states the crucial inequality together with its repeated-root version. Section~\ref{easy} treats the easy region via the optimal concavity of $\sigma_k$. Section~\ref{hard} develops the G{\aa}rding-root representation in the hard region and completes the proof of Theorem~\ref{thm-crucial-ineq}. Section~\ref{applications} discusses curvature and Hessian estimates.

\section{Preliminaries and the Crucial Concavity Inequality}\label{notation}
In this section, we collect the notation and elementary identities
used later, record two direct consequences of
Theorem~\ref{thm-crucial-ineq}, and introduce the normalization and
easy--hard decomposition used in the proof.
\begin{definition}
For $1\leq k\leq n$, the $k$-th elementary symmetric function
$\sigma_k$ is defined by
\begin{eqnarray*} 
\sigma_k(\lambda) = \sum _{1 \le i_1 < i_2 <\cdots<i_k\leq n}\lambda_{i_1}\lambda_{i_2}\cdots\lambda_{i_k},
 \qquad
\lambda=(\lambda_1,\ldots,\lambda_n)\in\mathbb{R}^n.
\end{eqnarray*}
\end{definition}
We use the convention that $\sigma_0=1$ and $\sigma_k =0$ for $k>n$. 

We write $\sigma_k(\lambda|i)$ for the elementary symmetric function
obtained by setting $\lambda_i=0$, and $\sigma_k(\lambda|ij)$ for the
one obtained by setting $\lambda_i=\lambda_j=0$.  Recall that the  G{\aa}rding cone is defined as
\begin{eqnarray*} 
\Gamma_k := \left\{ \lambda  \in \mathbb{R}^n \mid \sigma _i (\lambda ) > 0,~~\forall 1 \le i \le k \right\}.
\end{eqnarray*} 
The following properties will be used repeatedly.
\begin{proposition}
Let \(\lambda=(\lambda_1,\ldots,\lambda_n)\in\mathbb R^n\), and let
\(1\leq k\leq n\). Then
\begin{enumerate}
    \item $\sigma_k(\lambda)=\sigma_k(\lambda|i)+\lambda_i\sigma_{k-1}(\lambda|i), \quad \forall \,1\leq i\leq n$;
    \item $\sum\limits_{i=1}^n \lambda_i\sigma_{k-1}(\lambda|i)=k\sigma_{k}(\lambda)$;
    \item $\sum\limits_{i=1}^n \sigma_{k}(\lambda|i)=(n-k)\sigma_{k}(\lambda)$;
    \item If $\lambda \in \Gamma_k$ and $\lambda_1 \geq  \lambda_2 \geq \cdots \geq \lambda_n$, then
\begin{eqnarray}\label{n-k/k}
\lambda_i>-\frac{n-k}{k}\lambda_1.
\end{eqnarray}
\end{enumerate}

\end{proposition}
\begin{proof}
The first three identities follow directly from the definition of the
elementary symmetric functions. For the last assertion, see
\cite[Lemma~11]{RenWang2023}.
\end{proof}

Recall that the quadratic form
\(\mathcal Q_\gamma(\lambda;\xi)\) is defined in \eqref{Qgamma}.
It arises naturally when \(\log\lambda_{\max}\) is used as the
leading term of a maximum-principle test function; see
Section~\ref{applications}. Notice that the parameter range in
Theorem~\ref{thm-crucial-ineq} permits \(\gamma>1\) when
\(k>n/2\), whereas it requires \(\gamma<1\) when \(k\leq n/2\).

For completeness, we record the corresponding result when the largest
eigenvalue is not simple. It follows readily from
Theorem~\ref{thm-crucial-ineq} by a standard perturbation argument.
\begin{corollary}\label{cor:multiple-largest-eigenvalues}
Let \(n,k,\gamma\) satisfy the assumptions of
Theorem~\ref{thm-crucial-ineq}, and let \(\eta_*\) be the constant
given there. Suppose that \(\lambda\in\Gamma_k\) and that its largest
component has multiplicity \(m\geq2\); that is, $a=\lambda_1=\cdots=\lambda_m>\lambda_{m+1}\geq\cdots\geq\lambda_n$. Assume that
\begin{eqnarray*}
    0<\frac{F}{a^k}\leq\eta_*.
\end{eqnarray*}
Then, for every \(\xi\in\mathbb R^n\) satisfying
\begin{eqnarray*}
    \xi_2=\cdots=\xi_m=0,
\end{eqnarray*}
we have
\begin{eqnarray}\label{eq:multiple-largest-eigenvalues}
    \frac{2}{aF}
    \left(
        \sum_{i=1}^nF^{ii}\xi_i
    \right)^2
    -\gamma\frac{F^{11}}{a^2}\xi_1^2
    -\frac1a
    \sum_{i,j=1}^nF^{ii,jj}\xi_i\xi_j
    +\frac2a\sum_{p=m+1}^n
    \frac{F^{pp}}{a-\lambda_p}\,\xi_p^2
    \geq0.
\end{eqnarray}
\end{corollary}

\begin{proof}
For sufficiently small \(\varepsilon>0\), define
\begin{eqnarray*}
    \lambda^\varepsilon
    :=
    \left(
        a,
        \underbrace{a-\varepsilon,\ldots,a-\varepsilon}_{m-1\ {\rm times}},
        \lambda_{m+1},\ldots,\lambda_n
    \right).
\end{eqnarray*}
Then \(\lambda^\varepsilon\in\Gamma_k\) for sufficiently small
\(\varepsilon>0\), and $\lambda_1^\varepsilon>\lambda_2^\varepsilon\geq\cdots\geq\lambda_n^\varepsilon$.

Apply Theorem~\ref{thm-crucial-ineq} to
\((\lambda^\varepsilon,\xi)\) and using $\xi_2=\cdots=\xi_m=0$ gives the result.
\end{proof}

\begin{remark}\label{rmk-multiple}
Although the condition
\(\xi_2=\cdots=\xi_m=0\) is required in
Corollary~\ref{cor:multiple-largest-eigenvalues}, it imposes no
additional restriction in the geometric applications below.

Indeed, the standard viscosity formula for a multiple largest
eigenvalue in \cite{BCD2017} and the Codazzi equation give
\begin{eqnarray*}
    h_{\alpha\alpha;1}
    =
    h_{\alpha1;\alpha}=0,
    \qquad
    2\leq\alpha\leq m.
\end{eqnarray*}
Similarly, for Hessian equations, we have
\begin{eqnarray*}
    u_{\alpha\alpha1}
    =
    u_{\alpha1\alpha}
    =
    0,
    \qquad
    2\leq\alpha\leq m.
\end{eqnarray*}
Consequently, the vectors
\(\xi_i=h_{ii;1}\) and \(\xi_i=u_{ii1}\), respectively, satisfy the
hypothesis of Corollary~\ref{cor:multiple-largest-eigenvalues}.
\end{remark}

A direct corollary of Theorem \ref{thm-crucial-ineq} is the Jacobi inequality for $k>n/2$:
\begin{corollary}\label{cor:jacobi}
Let $n\geq3$, $n/2<k<n$, and let $u$ be a $C^4$ $k$-convex solution of the equation
\begin{eqnarray*}
    \sigma_k(D^2u)=1,\qquad \text{in } \Omega\subset\mathbb R^n.
\end{eqnarray*}
Suppose that $b:=\log\lambda_{\max}$ is sufficiently large. Then we have 
\begin{eqnarray}\label{jacobi}
\sum_{i,j=1}^nF^{ij}b_{ij}\geq \varepsilon \sum_{i,j=1}^nF^{ij}b_ib_j   
\end{eqnarray}
in the viscosity sense for any $0<\varepsilon<\frac{2k-n}{2k^2+n}$.
\end{corollary}
\begin{proof}
For any fixed point $x_0\in\Omega$, choose coordinates at $x_0$ such that
\begin{eqnarray*}
    D^2u
    =
    \operatorname{diag}(\lambda_1,\ldots,\lambda_n),
    \qquad
    a:=\lambda_1=\cdots=\lambda_m
    >\lambda_{m+1}\geq\cdots\geq\lambda_n.
\end{eqnarray*}
According to corollary \ref{cor:multiple-largest-eigenvalues} and remark \ref{rmk-multiple}, we may assume $m=1$.

Combining the second-derivative formula for the largest eigenvalue with the first and second derivatives of the equation, we obtain
\begin{eqnarray}\label{Fijbij-Fijbibj}
\sum_{i,j=1}^nF^{ij}b_{ij}-\varepsilon\sum_{i,j=1}^nF^{ij}b_ib_j=\mathcal{Q}_{1+\varepsilon}(\lambda;\xi)+\frac{1}{a^2}\sum_{p=2}^n\left(\frac{a+\lambda_p}{a-\lambda_p}-\varepsilon\right)F^{pp}u_{11p}^2.
\end{eqnarray}
Here $\mathcal{Q}_{1+\varepsilon}(\lambda;\xi)$ is the quadratic form defined in \eqref{Qgamma}, with $\xi_i:=u_{ii1}$. By Theorem~\ref{thm-crucial-ineq} and \eqref{n-k/k}, both terms on the right-hand side of \eqref{Fijbij-Fijbibj} are nonnegative whenever \(0<\varepsilon<(2k-n)/(2k^2+n)\). This proves \eqref{jacobi}.
\end{proof}

We now turn to the proof of
Theorem~\ref{thm-crucial-ineq}. By the homogeneity of \(\mathcal Q_\gamma\), it suffices to work under
the normalization
\begin{eqnarray*}
    1=a=\lambda_1=\lambda_{\max}
    >\lambda_2\geq\cdots\geq\lambda_n.
\end{eqnarray*}
Set
\begin{eqnarray*}
    P:=F^{11}=\sigma_{k-1}(\lambda|1).
\end{eqnarray*}
We divide the proof into two cases: the easy region
\begin{eqnarray*}
    \frac{P}{F}\leq 1+\frac{n}{2k^2},
\end{eqnarray*}
and the hard region
\begin{eqnarray*}
    \frac{P}{F}>1+\frac{n}{2k^2}.
\end{eqnarray*}
The quantity \(P/F\) describes the convexity of \(\lambda\).
In the easy region, \(\lambda\) is more convex and can be handled using
the optimal concavity of \(\sigma_k\), as shown in Section~\ref{easy}.
In the hard region, the situation becomes more complicated, and we need
to rewrite the equation in G{\aa}rding root coordinates and use the
concavity of the inverse G{\aa}rding-root function; see
Section~\ref{hard}.

\textbf{Throughout the proof of Theorem~\ref{thm-crucial-ineq}, we retain the above
normalization and simplicity assumption.}

\section{Optimal Concavity of \(\sigma_k\): the Easy Region}\label{easy}

Throughout this section, we retain the normalization
\(a=\lambda_1=1\) and consider the easy region defined by
\begin{eqnarray*}
    \frac{P}{F}
    \leq
    1+\frac{n}{2k^2}.
\end{eqnarray*}
In the easy region, \(\lambda\) is more convex. In fact, the term
\begin{eqnarray*}
    -\sum_{i,j=1}^n F^{ii,jj}\xi_i\xi_j
\end{eqnarray*}
alone is sufficient to control the principal part of the bad term. This control is provided by the optimal concavity of the \(\sigma_k\) operator; the case \(k=2\) was established
in \cite{Chen2013}.

Let \(H\) and \(g\) denote, respectively, the Hessian and the gradient
of \(\sigma_k\) at \(\lambda\):
\begin{eqnarray*}
    H:=D^2\sigma_k(\lambda),
    \qquad
    g:=D\sigma_k(\lambda).
\end{eqnarray*}
Then the term 
$$
\sum_{i,j=1}^n F^{ii,jj}x_ix_j
$$
can be written as 
$$
x^\mathsf T Hx.
$$

We begin with some basic properties of $H$ and $g$:
\begin{lemma}\label{lem:inertia}
For every \(\lambda\in\Gamma_k\), the matrix \(H\) is invertible and has inertia \((1,n-1)\); that is, it has one positive and \(n-1\) negative eigenvalues. Moreover,
\begin{eqnarray*}
    H\lambda=(k-1)g,\qquad
    H^{-1}g=\frac{\lambda}{k-1},\qquad
    g^{\mathsf T}H^{-1}g=\frac{kF}{k-1}.
\end{eqnarray*}
Furthermore, every diagonal entry of \(H^{-1}\) is negative:
\begin{eqnarray*}
    (H^{-1})_{ii}<0,\qquad 1\leq i\leq n.
\end{eqnarray*}
\end{lemma}

\begin{proof}
The Hessian of \(\sigma_k^{1/k}\) is negative semidefinite, with kernel
\(\operatorname{span}\{\lambda\}\). Euler's identity gives
\(g\cdot\lambda=kF>0\), and therefore
\(g^\perp\cap\operatorname{span}\{\lambda\}=\{0\}\). It follows that
\(D^2(\sigma_k^{1/k})\) is negative definite on \(g^\perp\). On this
hyperplane,
\begin{eqnarray*}
    D^2(\sigma_k^{1/k})
    =\frac{1}{k}\sigma_k^{\frac{1}{k}-1}H.
\end{eqnarray*}
Thus \(H\) has at least \(n-1\) negative directions. On the other
hand,
\begin{eqnarray*}
    \lambda^{\mathsf T}H\lambda=k(k-1)F>0.
\end{eqnarray*}
Therefore, \(H\) has exactly one positive direction and \(n-1\) negative
directions. In particular, \(H\) is invertible. The three identities in
the statement follow directly by homogeneity.

To prove \((H^{-1})_{ii}<0\), let \(H^{(i)}\) denote the matrix
obtained from \(H\) by deleting its \(i\)-th row and \(i\)-th column, and
let \(e_i\) be the \(i\)-th standard basis vector. Since \(e_i^\perp\)
and \(g^\perp\) are both \((n-1)\)-dimensional, their intersection has
dimension at least \(n-2\). It follows that \(H^{(i)}\) has at least
\(n-2\) negative directions.

For the fixed index \(i\), choose \(p\) and \(q\) so that
\(p,q,i\) are pairwise distinct. Set $v:=e_p+e_q$. Then
\begin{eqnarray*}
    v^{\mathsf T}Hv
    =2H_{pq}
    =2\sigma_{k-2}(\lambda|pq)>0.
\end{eqnarray*}
Hence \(H^{(i)}\) has exactly one positive direction and \(n-2\) negative
directions. Consequently,
\begin{eqnarray*}
    (H^{-1})_{ii}
    =\frac{\det H^{(i)}}{\det H}<0.
\end{eqnarray*}
\end{proof}
Set
\begin{eqnarray}\label{def-IkThetak}
\mathcal I_k(\lambda)=-F(H^{-1})_{11},\qquad \Theta_k(\lambda)=\frac{1}{1+k(k-1)\mathcal I_k(\lambda)}.    
\end{eqnarray}
We now present the exact maximum of the quadratic form
\(x^{\mathsf T}Hx\) subject to the linear constraint obtained by
differentiating the \(\sigma_k\) equation.
\begin{proposition}\label{prop:exact-concavity}
Let \(b,D\in\mathbb R\), and suppose that \(x\in\mathbb R^n\) satisfies
\(x_1=b\) and \(g\cdot x+D=0\). Then
\begin{eqnarray}\label{eq:exact-concavity}
    x^{\mathsf T}Hx
    \leq
    \frac{k-1}{kF}D^2
    -k(k-1)F\Theta_k
    \left(b+\frac{D}{kF}\right)^2,
\end{eqnarray}
where $\Theta_k$ is defined in \eqref{def-IkThetak}. Moreover, the right-hand side is the exact maximum of
\(x^{\mathsf T}Hx\) over this affine subspace.
\end{proposition}

\begin{proof}
Define
\begin{eqnarray*}
    \mathcal T
    =
    \bigl\{
        z\in\mathbb R^n:
        z_1=0,\ g\cdot z=0
    \bigr\}
    \subset g^\perp.
\end{eqnarray*}
This is precisely the tangent space of the affine constraint set. 

We first show that $H$ is strictly negative definite on $\mathcal T$. Set
\(u=H^{-1}g\). By Lemma~\ref{lem:inertia},
\begin{eqnarray*}
    u^{\mathsf T}Hu
    =g^{\mathsf T}H^{-1}g
    =\frac{kF}{k-1}>0.
\end{eqnarray*}
For every \(t\in g^\perp\), we have
\begin{eqnarray*}
    t^{\mathsf T}Hu=t^{\mathsf T}g=0.
\end{eqnarray*}
Thus \(g^\perp\) is the \(H\)-orthogonal complement of the positive
direction \(u\). As \(H\) has inertia \((1,n-1)\), its restriction to
\(g^\perp\) is negative definite. In particular,
\(\mathcal T\subset g^\perp\) implies
\begin{eqnarray*}
    H|_{\mathcal T}<0.
\end{eqnarray*}
Consequently, \(x^{\mathsf T}Hx\) is a strictly concave quadratic
function on the given affine subspace and admits a unique maximizer
\(x_*\).

The first-order maximum condition gives, for every
\(t\in\mathcal T\),
\begin{eqnarray*}
    \left.
    \frac{d}{ds}
    (x_*+st)^{\mathsf T}H(x_*+st)
    \right|_{s=0}
    =0.
\end{eqnarray*}
Hence \(t^{\mathsf T}Hx_*=0\) for every \(t\in\mathcal T\), and therefore
\begin{eqnarray*}
    Hx_*\in\mathcal T^\perp
    =\operatorname{span}\{e_1,g\}.
\end{eqnarray*}
It follows that
\begin{eqnarray*}
    x_*\in\operatorname{span}\{H^{-1}e_1,\lambda\}.
\end{eqnarray*}
Solving the constraints
\(e_1\cdot x_*=b\) and \(g\cdot x_*+D=0\) yields
\begin{eqnarray*}
    x_*
    =
    -\frac{D}{kF}\lambda
    -\frac{b+D/(kF)}{L_1}w,
\end{eqnarray*}
where
\begin{eqnarray*}
    L_1
    =
    \frac{1}{k(k-1)F}-(H^{-1})_{11}>0,
    \qquad
    w
    =
    H^{-1}e_1-\frac{1}{k(k-1)F}\lambda.
\end{eqnarray*}
Substituting \(x_*\) into \(x^{\mathsf T}Hx\) gives the
right-hand side of \eqref{eq:exact-concavity}.
\end{proof}
To proceed, we need an estimate for \(\Theta_k\), or equivalently for
\(\mathcal I_k\), in a form suited to our argument (see Proposition \ref{prop:shape-modulus}). Fix \(i\geq2\). Define
\begin{eqnarray}
 p_i(\lambda):=\sigma_{k-1}(\lambda|i),
 \qquad
 q_i(\lambda):=\sigma_k(\lambda|i).
 \label{eq:pi-qi}
\end{eqnarray}
Both quantities are independent of \(\lambda_i\), and
\begin{eqnarray}
 F=\lambda_i p_i+q_i,
 \qquad
 \varepsilon_i:=\frac{F}{p_i}
 =\lambda_i+\frac{q_i}{p_i}>0.
 \label{eq:epsi}
\end{eqnarray}

Set
\begin{eqnarray*}
 E_i:=\{w\in\mathbb R^n:w_i=0\},
 \qquad
 V_i:=\{w\in E_i:Dp_i\cdot w=0\}.
\end{eqnarray*}
We still view \(p_i\) and \(q_i\) as functions on \(\mathbb R^n\). Then their first derivatives with respect to \(\lambda_i\) vanish,
and the \(i\)-th rows and columns of their Hessians are zero. On \(V_i\), define
\begin{eqnarray}
 \mathcal C_i:=-D^2p_i,
 \qquad
 \mathcal R_i:=\frac{q_i}{p_i}D^2p_i-D^2q_i.
 \label{eq:Ci-Ri}
\end{eqnarray}
The strict concavity of \(\sigma_{k-1}^{1/(k-1)}\) implies that
\(\mathcal C_i\) is positive definite on \(V_i\). On the other hand,
for \(w\in V_i\),
\begin{eqnarray*}
 D^2\!\left(\frac{q_i}{p_i}\right)[w,w]
 =
 \frac1{p_i}
 \left(
 D^2q_i-\frac{q_i}{p_i}D^2p_i
 \right)[w,w].
\end{eqnarray*}
Hence, by the concavity of the Hessian quotient
\(\sigma_k/\sigma_{k-1}\), we have \(\mathcal R_i\geq0\) on \(V_i\).

\

We first derive a variational energy characterization of
\(\mathcal I_k\). More precisely, Lemma~\ref{lem:energy-original}
expresses the inverse-Hessian quantity \(-F(H^{-1})_{11}\) as a variational supremum over the constrained subspace \(V_i\). Equivalently, \(\mathcal I_k/p_i\) is the squared dual norm of the
coordinate functional \(w\mapsto w_1\) with respect to the
positive-definite quadratic form
\(\mathcal C_i+\varepsilon_i^{-1}\mathcal R_i\); the associated
minimum-energy problem over all \(w\in V_i\) satisfying \(w_1=1\)
has value \(p_i/\mathcal I_k\). A key
feature of this representation is that the dependence on
\(\lambda_i\) is isolated in the single scalar parameter
\(\varepsilon_i\), while \(\mathcal C_i\) and \(\mathcal R_i\) depend
only on the remaining coordinates. This formulation makes coordinate
monotonicity transparent and provides a convenient framework for the
limiting arguments below.
\begin{lemma}\label{lem:energy-original}
For every \(i\geq2\),
\begin{eqnarray}
 \mathcal I_k(\lambda)=
 p_i
 \sup_{0\ne w\in V_i}
 \frac{w_1^2}
 {\displaystyle
  \mathcal C_i[w,w]+\frac1{\varepsilon_i}\mathcal R_i[w,w]}.
 \label{eq:energy-original}
\end{eqnarray}
\end{lemma}

\begin{proof}
Set
\begin{eqnarray*}
 Y:=H^{-1}e_1,
 \qquad
 \alpha:=Y_i,
 \qquad
 z:=Y-\alpha e_i.
\end{eqnarray*}
Since \(Dp_i=He_i\), a direct calculation gives
\begin{eqnarray*}
 Dp_i\cdot z
 =
 He_i\cdot(Y-\alpha e_i) 
 =
 e_i^{\mathsf T}e_1-\alpha e_i^{\mathsf T}He_i
 =0.
\end{eqnarray*}
Thus \(z\in V_i\).

For any \(w\in V_i\), we have
\begin{eqnarray*}
 w_1
 =
 w^{\mathsf T}HY 
 =
 \alpha\,w^{\mathsf T}He_i+H[w,z]
 =H[w,z],
\end{eqnarray*}
where we have used \(He_i=Dp_i\) and \(Dp_i\cdot w=0\). Since
\(w,z\in V_i\), the equation \(F=\lambda_i p_i+q_i\) gives
\begin{eqnarray*}
 H[w,z]
 =
 \lambda_iD^2p_i[w,z]+D^2q_i[w,z] 
 =
 -\varepsilon_i
 \left(
 \mathcal C_i+\frac1{\varepsilon_i}\mathcal R_i
 \right)[w,z].
\end{eqnarray*}
Consequently,
\begin{eqnarray}
 \left(
 \mathcal C_i+\frac1{\varepsilon_i}\mathcal R_i
 \right)[z,w]
 =
 -\frac1{\varepsilon_i}w_1
 \qquad (w,z\in V_i).
 \label{eq:Riesz-original}
\end{eqnarray}

Moreover, the quadratic form
\begin{eqnarray*}
 \mathcal A_i
 :=
 \mathcal C_i+\frac1{\varepsilon_i}\mathcal R_i
\end{eqnarray*}
is positive definite on \(V_i\). By the Riesz representation theorem, there exists a unique
\(u\in V_i\) such that
\begin{eqnarray*}
    \mathcal A_i[u,w]=w_1
    \qquad\text{for every }w\in V_i.
\end{eqnarray*}
Since \(i\ne1\),
\begin{eqnarray*}
 (H^{-1})_{11}
 =
 e_1\cdot Y
 =e_1\cdot z
 =z_1.
\end{eqnarray*}
Therefore, the Cauchy--Schwarz inequality gives
\begin{eqnarray*}
 -(H^{-1})_{11}
 =
 -z_1
 =\frac1{\varepsilon_i}\,u_1 
 =
 \frac1{\varepsilon_i}
 \sup_{0\ne w\in V_i}
 \frac{w_1^2}{\mathcal A_i[w,w]},
\end{eqnarray*}
which proves
\eqref{eq:energy-original}.
\end{proof}
A direct corollary of Lemma \ref{lem:energy-original} is the monotonicity of $\mathcal I_k$.
\begin{corollary}
\label{cor:coordinate-monotonicity}
Fix \(i\geq2\) and all the remaining coordinates. Then, for every
\(s\geq0\),
\begin{eqnarray*}
 \mathcal I_k(\lambda+s e_i)\geq\mathcal I_k(\lambda).
\end{eqnarray*}
\end{corollary}

\begin{proof}
In \eqref{eq:energy-original}, the quantities \(p_i,q_i\), and \(V_i\),
as well as the forms \(\mathcal C_i\) and \(\mathcal R_i\), are
unchanged as \(s\) varies, whereas
\(\varepsilon_i\) is replaced by \(\varepsilon_i+s\). Since
\(\mathcal R_i\geq0\) on \(V_i\), the monotonicity follows.
\end{proof}

To analyze the energy characterization in
Lemma~\ref{lem:energy-original} as \(F\to0\),
we need a compactness principle that allows us to pass to the limit in
the corresponding variational quotients. Since
\(\varepsilon_i\to0\), an approximate maximizing
sequence can yield a nonzero limit only if \(\mathcal R_i[w,w]\) tends to zero. Hence every such limiting direction lies in the kernel of the limiting quadratic form. At the same time, both
the constraint subspaces and the quadratic forms may vary along the
sequence. The following lemma makes this limiting behavior
precise and provides the required upper bound.
\begin{lemma}\label{lem:penalty}
Let \(E\) be a fixed finite-dimensional Euclidean space, and suppose
\begin{eqnarray*}
 u_j\longrightarrow u_*\ne0,
 \qquad
 V_j:=E\cap u_j^\perp,
 \qquad
 V_*:=E\cap u_*^\perp,
\end{eqnarray*}
and
\begin{eqnarray*}
 \mathcal C_j\longrightarrow\mathcal C_*,
 \qquad
 \mathcal R_j\longrightarrow\mathcal R_*,
 \qquad
 \varepsilon_j\longrightarrow0^+,
 \qquad
 l_j\longrightarrow l_*.
\end{eqnarray*}
Assume that
\begin{eqnarray*}
 \mathcal C_*>0\quad\text{on }V_*,
 \qquad
 \mathcal R_j\geq0\quad\text{on }V_j,
 \qquad
 \mathcal R_*\geq0\quad\text{on }V_*.
\end{eqnarray*}
If
\begin{eqnarray*}
 M_j
 :=
 \sup_{0\ne w\in V_j}
 \frac{l_j(w)^2}
 {\mathcal C_j[w,w]+\mathcal R_j[w,w]/\varepsilon_j},
\end{eqnarray*}
then
\begin{eqnarray}
 \limsup_{j\to\infty}M_j
 \leq
 \sup_{\substack{
     0\ne w\in V_*\\
     \mathcal R_*[w,w]=0
 }}
 \frac{l_*(w)^2}{\mathcal C_*[w,w]}.
 \label{eq:moving-penalty}
\end{eqnarray}
If ~\(\ker(\mathcal R_*|_{V_*})=\{0\}\), the right-hand side is
understood to be \(0\).
\end{lemma}

\begin{proof}
The positive definiteness of \(\mathcal C_*\) on \(V_*\), together
with the convergence assumptions, implies that \(\mathcal C_j\) is
uniformly positive definite on \(V_j\) for all sufficiently large
\(j\).

If \(\limsup_j M_j=0\), there is nothing to prove. Otherwise, we can choose \(\widehat w_j\in V_j\setminus\{0\}\) such that
\begin{eqnarray*}
 \frac{l_j(\widehat w_j)^2}
 {\displaystyle
  \mathcal C_j[\widehat w_j,\widehat w_j]
  +\frac1{\varepsilon_j}
   \mathcal R_j[\widehat w_j,\widehat w_j]}
 \geq
 M_j-\frac1j.
\end{eqnarray*}
By the zero-homogeneity of this quotient in \(\widehat w_j\),
we may rescale it to obtain \(w_j\in V_j\) satisfying
\begin{eqnarray}
 \mathcal C_j[w_j,w_j]
 +\frac1{\varepsilon_j}\mathcal R_j[w_j,w_j]
 =
 1,
 \label{eq:moving-normalization}
\end{eqnarray}
and therefore
\begin{eqnarray}
 l_j(w_j)^2
 \geq
 M_j-\frac1j.
 \label{eq:moving-almost-max}
\end{eqnarray}

The uniform positive definiteness of \(\mathcal C_j\), together with
\eqref{eq:moving-normalization}, implies that \((w_j)\) is bounded.
After passing to a subsequence, let \(w_j\to w_*\in E\). Since
\(u_j\cdot w_j=0\), we obtain \(u_*\cdot w_*=0\), and hence
\(w_*\in V_*\).

On the other hand, since \(\mathcal R_j\geq0\) on \(V_j\), equation
\eqref{eq:moving-normalization} gives
\begin{eqnarray*}
 0\leq\mathcal R_j[w_j,w_j]
 \leq\varepsilon_j\longrightarrow0.
\end{eqnarray*}
Together with \(\mathcal R_j\to\mathcal R_*\) and \(w_j\to w_*\),
we have $\mathcal R_*[w_*,w_*]=0$. Thus \(w_*\in\ker\mathcal R_*\). Similarly, $\mathcal C_*[w_*,w_*]\leq1$.

If \(w_*=0\), then $l_j(w_j)\longrightarrow0$, and \eqref{eq:moving-almost-max} yields \(\limsup_jM_j=0\).

If \(w_*\ne0\), \eqref{eq:moving-almost-max} and the convergence
of \(l_j\) give 
\begin{eqnarray*}
 \limsup_{j\to\infty}M_j
 \leq
 l_*(w_*)^2
 \leq
 \frac{l_*(w_*)^2}{\mathcal C_*[w_*,w_*]}.
\end{eqnarray*}
Since \(w_*\in V_*\) and \(\mathcal R_*[w_*,w_*]=0\), this proves the lemma.
\end{proof}
We are now ready to prove the desired estimate for \(\Theta_k\).
By the definition of \(\Theta_k\), this estimate is equivalent to the
following upper bound for \(\mathcal I_k\).
\begin{proposition}\label{prop:boundary}
Suppose that
\begin{eqnarray*}
 \lambda^{(j)}\in\Gamma_k,
 \qquad
 \lambda_1^{(j)}=\max_i\lambda_i^{(j)}=1,
 \qquad
 F_j:=\sigma_k(\lambda^{(j)})\longrightarrow0.
\end{eqnarray*}
Then
\begin{eqnarray}
 \limsup_{j\to\infty}\mathcal I_k(\lambda^{(j)})
 \leq
 \frac{k-2}{k-1}.
 \label{eq:boundary-result}
\end{eqnarray}
\end{proposition}

\begin{proof}
Since all coordinates are bounded, up to a subsequence, we may assume that
\begin{eqnarray*}
 \lambda^{(j)}\longrightarrow
 \lambda^*\in\overline{\Gamma}_k,
 \qquad
 \sigma_k(\lambda^*)=0.
\end{eqnarray*}
It suffices to prove the asserted upper bound for $\lambda^{(j)}$.

\medskip
\noindent
\textbf{Case 1: \(\sigma_{k-1}(\lambda^*)>0\).}

For each \(j\), choose an index at which \(\lambda_i^{(j)}\) is
minimal among \(i\in\{2,\ldots,n\}\). After passing to a further
subsequence, we may assume that this index is fixed, denoted by
\(i_0\geq2\). Set
\begin{eqnarray*}
 t_j:=\lambda_{i_0}^{(j)}
 \longrightarrow t_*:=\lambda_{i_0}^*.
\end{eqnarray*}
Since \(\sigma_k(\lambda^*)=0\), we have \(t_*\leq0\).

For simplicity, set
\begin{eqnarray*}
 p_j:=p_{i_0}(\lambda^{(j)}),
 \qquad
 p_*:=p_{i_0}(\lambda^*),
 \qquad
 q_j:=q_{i_0}(\lambda^{(j)}),
 \qquad
 \varepsilon_j:=\frac{F_j}{p_j}.
\end{eqnarray*}
Since
\begin{eqnarray*}
 \sigma_{k-1}(\lambda^*)
 =
 p_*+t_*\sigma_{k-2}(\lambda^*| i_0)>0,
\end{eqnarray*}
we have \(p_*>0\), and hence \(\varepsilon_j\to0^+\).

\medskip
\noindent
\textbf{Subcase 1.1: \(t_*<0\).}
In this subcase, \(\lambda^*| i_0\in\Gamma_k\).

On \(V_{i_0}^*\), we have
\begin{eqnarray*}
 \mathcal R_{i_0}^*
 =
 -p_*D^2\!\left(\frac{q_{i_0}}{p_{i_0}}\right).
\end{eqnarray*}
Restricted to
\(E_{i_0}=\{w\in\mathbb R^n:w_{i_0}=0\}\), the Hessian
\(D^2(q_{i_0}/p_{i_0})(\lambda^*)\) has a one-dimensional kernel
spanned by the radial direction
\(\lambda^*-t_*e_{i_0}\). Since
\(Dp_{i_0}(\lambda^*)\cdot(\lambda^*-t_*e_{i_0})
=(k-1)p_*>0\), this direction does not belong to \(V_{i_0}^*\).
Therefore, \(\mathcal R_{i_0}^*\) is positive definite on
\(V_{i_0}^*\). If \(M_j\) denotes the corresponding supremum, 
Lemma~\ref{lem:penalty} then gives \(M_j\to0\). Since
\(p_j\to p_*>0\), it follows that
\begin{eqnarray*}
 \mathcal I_k(\lambda^{(j)})
 =
 p_jM_j\longrightarrow0.
\end{eqnarray*}

\medskip
\noindent
\textbf{Subcase 1.2: \(t_*=0\).}
Since \(t_j\) is the smallest coordinate among the indices
\(2,\ldots,n\), we have \(\lambda^*\geq0\). The conditions $\sigma_{k-1}(\lambda^*)>0$ and $\sigma_k(\lambda^*)=0$ imply that \(\lambda^*\) has exactly \(k-1\) positive coordinates.
Keep the zero coordinate \(i_0\geq2\) fixed and work on the coordinate
subspace \(E_{i_0}\). Write
\begin{eqnarray*}
 I_+:=\{a:\lambda_a^*>0\},
 \qquad
 I_0:=\{1,\ldots,n\}\setminus I_+,
 \qquad
 I_1:=I_0\setminus\{i_0\}.
\end{eqnarray*}
Then \(|I_+|=k-1\). Let \(\mu^*\in\mathbb R^{k-1}\) be any
ordering of these \(k-1\) positive numbers. Since only symmetric functions of \(\mu^*\) appear below, none of the
subsequent expressions depends on the chosen ordering. For
any \(a\ne b\in I_+\) and \(\alpha\ne\beta\in I_1\), a direct
calculation yields
\begin{eqnarray*}
 p_{i_0}(\lambda^*)
 =\sigma_{k-1}(\mu^*),\qquad
 \partial_a p_{i_0}(\lambda^*)
 =\sigma_{k-2}(\mu^*| a),\qquad
 \partial_\alpha p_{i_0}(\lambda^*)
 =\sigma_{k-2}(\mu^*).
\end{eqnarray*}
Consequently, \(w\in V_{i_0}^*\) if and only if
\begin{eqnarray}
 0
 =
 Dp_{i_0}(\lambda^*)\cdot w =
 \sum_{a\in I_+}\sigma_{k-2}(\mu^*| a)w_a
 +\sigma_{k-2}(\mu^*)\sum_{\alpha\in I_1}w_\alpha.
 \label{winVi0*}
\end{eqnarray}

We next compute \(\mathcal R_{i_0}^*[w,w]\) for
\(w\in V_{i_0}^*\). 

Since
\(q_{i_0}(\lambda^*)=\sigma_k(\lambda^*|i_0)=0\), we have
\(\mathcal R_{i_0}^*=-D^2q_{i_0}\) at this point. Moreover, its only nonzero mixed second
derivatives are
\begin{eqnarray*}
 \partial_{a\alpha}q_{i_0}(\lambda^*)
 =
 \sigma_{k-2}(\mu^*| a),
 \qquad
 \partial_{\alpha\beta}q_{i_0}(\lambda^*)
 =\sigma_{k-2}(\mu^*).
\end{eqnarray*}
Therefore,
\begin{eqnarray*}
 D^2q_{i_0}[w,w]
 &=&
 2\sum_{a\in I_+}\sigma_{k-2}(\mu^*| a)w_a
 \sum_{\alpha\in I_1}w_\alpha
 +\sum_{\alpha\ne\beta\in I_1}
 \sigma_{k-2}(\mu^*)w_\alpha w_\beta \\
 &=&
 \sigma_{k-2}(\mu^*)
 \left(
 -2\left(\sum_{\alpha\in I_1}w_\alpha\right)^2
 +\sum_{\alpha\ne\beta\in I_1}w_\alpha w_\beta
 \right) \\
 &=&
 -\sigma_{k-2}(\mu^*)
 \left(
 \sum_{\alpha\in I_1}w_\alpha^2
 +\left(\sum_{\alpha\in I_1}w_\alpha\right)^2
 \right),
\end{eqnarray*}
where the second equality follows from \eqref{winVi0*}. Hence
\(w\in\ker\mathcal R_{i_0}^*\) if and only if
\(w_\alpha=0\) for every \(\alpha\in I_1\) and
\begin{eqnarray}
 \sum_{a\in I_+}\sigma_{k-2}(\mu^*| a)w_a=0.
 \label{ker-old}
\end{eqnarray}
For convenience, define \(x\in\mathbb R^{k-1}\) by
\(x_a:=w_a/\lambda_a^*\). Then the kernel condition
\eqref{ker-old} is equivalent to
\begin{eqnarray}
 \sum_{a\in I_+}x_a=\sigma_1(x)=0.
 \label{ker-new}
\end{eqnarray}
Since \(w_\alpha\equiv0\) on the kernel, a direct calculation gives
\begin{eqnarray*}
 \mathcal C_{i_0}^*[w,w]
 =
 -D^2p_{i_0}[w,w]
 =-2\sigma_{k-1}(\mu^*)\sigma_2(x)
 =-2p_*\sigma_2(x).
\end{eqnarray*}
Moreover, \(\lambda_1^*=1\), and hence \(w_1=x_1\). The
Cauchy--Schwarz inequality yields
\begin{eqnarray}
 p_*
 \sup_{0\ne w\in\ker\mathcal R_{i_0}^*}
 \frac{w_1^2}{\mathcal C_{i_0}^*[w,w]}
 \leq
 \frac{k-2}{k-1},
 \label{rank=k-1}
\end{eqnarray}
with equality when \(x_a=-x_1/(k-2)\) for every
\(a\in I_+\setminus\{1\}\). The desired conclusion follows from
Lemma~\ref{lem:penalty}.

\medskip
\noindent
\textbf{Case 2: \(\sigma_{k-1}(\lambda^*)=0\).}

We first claim that \(\lambda^*\geq0\). Suppose to the contrary that
\(\lambda_{i_0}^*=-s<0\) and set
\(\mu=\lambda^*| i_0\). Then \(\mu\in\overline{\Gamma}_{k-1}\). Since
\begin{eqnarray*}
    \sum_{i=1}^n\sigma_{k-1}(\lambda^*|i)
    =
    (n-k+1)\sigma_{k-1}(\lambda^*)=0,
\end{eqnarray*}
we have
\(\sigma_{k-1}(\mu)=0\).

If \(\sigma_l(\mu)=0\) for some \(2\leq l\leq k-1\), then
\begin{eqnarray*}
    0
    \leq
    \sigma_l(\lambda^*)
    =
    \sigma_l(\mu)-s\sigma_{l-1}(\mu)
    =
    -s\sigma_{l-1}(\mu)
    \leq0.
\end{eqnarray*}
Hence \(\sigma_{l-1}(\mu)=0\). By iteration, we obtain
\(\sigma_1(\mu)=0\). Consequently,
\begin{eqnarray*}
    \sigma_1(\lambda^*)
    =
    \sigma_1(\mu)-s=-s<0,
\end{eqnarray*}
which is a contradiction. Therefore
\(\lambda^*\geq0\).

Set
\begin{eqnarray*}
 I_+:=\{a:\lambda_a^*>0\},
 \qquad
 I_0:=\{1,\ldots,n\}\setminus I_+,
 \qquad
 r:=|I_+|.
\end{eqnarray*}
Since \(\lambda_1^*=1\) and \(\sigma_{k-1}(\lambda^*)=0\), we have
\begin{eqnarray*}
 1\leq r\leq k-2.
\end{eqnarray*}
Define
\begin{eqnarray}
 \delta_j:=\max_{\alpha\in I_0}\lambda_\alpha^{(j)}.
 \label{eq:delta}
\end{eqnarray}
Each \(\lambda^{(j)}\) has at least \(k\) positive coordinates,
whereas \(I_+\) has only \(r<k\) indices. Hence \(\delta_j>0\) and \(\delta_j\to0^+\).

Increase each coordinate indexed by \(I_0\) to \(\delta_j\), while leaving the remaining coordinates unchanged, and
define
\begin{eqnarray*}
 \widetilde\lambda_a^{(j)}
 =
 \lambda_a^{(j)}
 \quad(a\in I_+),
 \qquad
 \widetilde\lambda_\alpha^{(j)}
 =\delta_j
 \quad(\alpha\in I_0).
\end{eqnarray*}
By the monotonicity corollary \ref{cor:coordinate-monotonicity},
\begin{eqnarray}
 \mathcal I_k(\lambda^{(j)})
 \leq
 \mathcal I_k(\widetilde\lambda^{(j)}).
 \label{eq:raise}
\end{eqnarray}

For \(x=(x_a)_{a\in I_+}\) and
\(y=(y_\alpha)_{\alpha\in I_0}\), define
\begin{eqnarray}
 f_\delta(x,y)
 :=
 \delta^{-(k-r)}\sigma_k(x,\delta y)
 =
 \sum_{s=0}^{r}
 \delta^{r-s}\sigma_s(x)\sigma_{k-s}(y).
 \label{eq:Phi-delta}
\end{eqnarray}
A direct calculation gives
\begin{eqnarray}
 \mathcal I_k(x,\delta y)
 =
 -f_\delta(x,y)
 \left[
 \left(D^2f_\delta(x,y)\right)^{-1}
 \right]_{11}
 .
 \label{eq:scale-invariance}
\end{eqnarray}
Thus, by passing from \(\mathcal I_k\) to \(f_\delta\), we replace
the degenerating variables \((x,\delta y)\) with the nondegenerate
variables \((x,y)\). Moreover, as \(\delta\to0\), \(f_\delta\)
converges locally in \(C^2\) to
\begin{eqnarray}
 f_*(x,y)
 :=
 \sigma_r(x)\sigma_{k-r}(y).
 \label{eq:Phi-limit}
\end{eqnarray}

For \(\widetilde\lambda^{(j)}\), we have
\(y=\mathbf 1_{I_0}\) and
\(x_a=\lambda_a^{(j)}\to\lambda_a^*\). As shown below, the limiting
Hessian is invertible at this point; hence, by local \(C^2\)
convergence, the inverses of the Hessians also converge. It remains to
prove
\begin{eqnarray}
 \lim_{j\to\infty}\mathcal I_k(\widetilde\lambda^{(j)})
 =
 \frac{k-2}{k-1}(\lambda_1^*)^2
 =\frac{k-2}{k-1}.
 \label{keyineq-r<k-1}
\end{eqnarray}
Define the matrix
\begin{eqnarray*}
    A:=\frac{D^2f_*}{f_*}.
\end{eqnarray*}
For the computations below, we fix
\(y=\mathbf 1_{I_0}\). A direct calculation
gives
\begin{eqnarray*}
 A_{aa}=0,
 \qquad
 A_{ab}=\frac1{x_ax_b}
 \quad(a\ne b,\ a,b\in I_+),\qquad A_{a\alpha}
 =
 \frac{k-r}{(n-r)x_a}
 \quad(a\in I_+,\ \alpha\in I_0),
\end{eqnarray*}
and
\begin{eqnarray*}
 A_{\alpha\alpha}=0,
 \qquad
 A_{\alpha\beta}
 =
 \frac{(k-r)(k-r-1)}{(n-r)(n-r-1)}
 \quad(\alpha\ne\beta,\ \alpha,\beta\in I_0).
\end{eqnarray*}

We first prove that \(A\) is invertible. Suppose that
\begin{eqnarray*}
    A(u,z)=0,
\end{eqnarray*}
where
\begin{eqnarray*}
    u=(u_a)_{a\in I_+},
    \qquad
    z=(z_\alpha)_{\alpha\in I_0},
\end{eqnarray*}
and set
\begin{eqnarray*}
    s:=\sum_{a\in I_+}\frac{u_a}{x_a}.
\end{eqnarray*}
Subtracting any two \(I_0\)-component equations shows that all
\(z_\alpha\) have a common value, denoted by \(c\). Each
\(I_0\)-component equation then reduces to
\begin{eqnarray*}
    s+(k-r-1)c=0.
\end{eqnarray*}
On the other hand, the \(a\)-th \(I_+\)-component equation gives
\begin{eqnarray*}
    \frac{u_a}{x_a}
    =
    s+(k-r)c
    =
    c.
\end{eqnarray*}
Summing over \(a\in I_+\), we obtain
\begin{eqnarray*}
    s=rc.
\end{eqnarray*}
Comparing this identity with
\(s=-(k-r-1)c\) yields \(c=0\), and consequently \(z=0\), \(s=0\), and \(u=0\).
Therefore \(\ker A=\{0\}\), so \(A\) is invertible.

We next compute \((A^{-1})_{11}\). Let
\begin{eqnarray*}
    (u,z)=A^{-1}e_1.
\end{eqnarray*}
Applying the same argument as above, we find that all
\(z_\alpha\) have a common value \(c\) and that
\begin{eqnarray*}
    (k-1)c=x_1,
    \qquad
    \frac{u_a}{x_a}
    =
    c-x_a\delta_{a1}.
\end{eqnarray*}
In particular,
\begin{eqnarray*}
    u_1
    =
    x_1(c-x_1)
    =
    -\frac{k-2}{k-1}x_1^2.
\end{eqnarray*}
Since \(u_1=(A^{-1})_{11}\), we conclude that
\begin{eqnarray*}
    (A^{-1})_{11}
    =
    -\frac{k-2}{k-1}x_1^2.
\end{eqnarray*}
Finally, \(D^2f_*=f_*A\), and hence
\begin{eqnarray*}
    -f_*
    \left[
        \left(D^2f_*^{-1}\right)_{11}
    \right]
    =
    -(A^{-1})_{11}
    =
    \frac{k-2}{k-1}x_1^2.
\end{eqnarray*}
This proves \eqref{keyineq-r<k-1}.
\end{proof}

\begin{proposition}\label{prop:shape-modulus}
For each $\varepsilon>0$, there exists a small constant $\eta=\eta(n,k,\varepsilon)>0$ such that whenever 
$$
F=\sigma_k(\lambda)<\eta,
$$
we have
\begin{eqnarray}\label{eq:shape-modulus}
    \Theta_k(\lambda)
    \geq
    \frac{1}{
        (k-1)^2+\varepsilon
    }.
\end{eqnarray}
\end{proposition}

\begin{proof}
When \(k=2\), we have the exact identity
\begin{eqnarray}\label{eq:Delta2-exact}
    \mathcal I_2
    =
    \frac{n-2}{n-1}F.
\end{eqnarray}
The assertion follows.

For \(k\geq3\), Proposition~\ref{prop:boundary} implies that, for all
normalized \(\lambda\in\Gamma_k\),
\begin{eqnarray*}
    \mathcal I_k(\lambda)
    \leq
    \frac{k-2}{k-1}+o(1)
    \qquad\text{as }F\downarrow0.
\end{eqnarray*}
Therefore, for every \(\varepsilon>0\), if \(F\) is sufficiently
small, then
\begin{eqnarray*}
    \Theta_k(\lambda)^{-1}
    =
    1+k(k-1)\mathcal I_k(\lambda)
    \leq
    (k-1)^2+\varepsilon.
\end{eqnarray*}
This proves Proposition \ref{prop:shape-modulus}.
\end{proof}
Based on Proposition~\ref{prop:shape-modulus}, we are now ready
to prove Theorem~\ref{thm-crucial-ineq} in the easy region.

Recall the quadratic form
\begin{eqnarray*}
\mathcal Q_{\gamma}(\lambda;\xi)
=
\frac{2}{F}\left(\sum_{i=1}^n F^{ii}\xi_i\right)^2
-\gamma P\xi_1^2
-H[\xi,\xi]
+2\sum_{p=2}^n
\frac{F^{pp}}{1-\lambda_p}\,\xi_p^2.   
\end{eqnarray*}
\begin{lemma}\label{lem:easy}
There exists a small constant $\eta_{\mathrm{easy}}=\eta_{\mathrm{easy}}(n,k,\gamma)>0$ such that, whenever
\begin{eqnarray*}
    0<F\leq\eta_{\mathrm{easy}},
    \qquad
    \frac{P}{F}
    \leq
    1+\frac{n}{2k^2},
\end{eqnarray*}
we have
\begin{eqnarray}
\label{eq:easy-region}
    \mathcal Q_{\gamma}(\lambda;\xi)\geq0
    \qquad \forall \xi\in\mathbb R^n.
\end{eqnarray}
\end{lemma}

\begin{proof}
In Proposition~\ref{prop:exact-concavity}, take
\begin{eqnarray*}
    x=\xi,
    \qquad
    b=\xi_1,
    \qquad
    D=-g\cdot\xi=-\sum_{i=1}^n F^{ii}\xi_i.
\end{eqnarray*}
Since
\begin{eqnarray*}
    g\cdot\xi+D=0,
\end{eqnarray*}
Proposition~\ref{prop:exact-concavity} gives
\begin{eqnarray*}
    -H[\xi,\xi]
    \geq
    -\frac{k-1}{kF}(g\cdot\xi)^2
    +
    k(k-1)F\Theta_k
    \left(
        \xi_1-\frac{g\cdot\xi}{kF}
    \right)^2.
\end{eqnarray*}

For $\Theta>0$, define the function
\begin{eqnarray*}
    \mathfrak c_{k}(\Theta)
    :=
    \frac{
        k(k-1)(k+1)\Theta
    }{
        k+1+(k-1)\Theta
    }.
\end{eqnarray*}
Substituting the preceding estimate into \(\mathcal Q_\gamma\) and
completing the square with respect to \((g\cdot\xi)/F\), we obtain
\begin{eqnarray}\label{Qgamma-ck2}
\mathcal Q_\gamma(\lambda;\xi)
&\geq&
F\,
\frac{
    k+1+(k-1)\Theta_k
}{k}
\left[
    \frac{g\cdot\xi}{F}
    -
    \frac{
        k(k-1)\Theta_k
    }{
        k+1+(k-1)\Theta_k
    }\xi_1
\right]^2
\\
&&+
F\xi_1^2
\left[
    \mathfrak c_{k}(\Theta_k)
    -\gamma\frac{P}{F}
\right]
+
2\sum_{p=2}^n
\frac{F^{pp}}{1-\lambda_p}\,\xi_p^2.
\end{eqnarray}

The function \(\mathfrak c_{k}\) is continuous and strictly increasing
on \((0,\infty)\). Moreover, the assumption on \(\gamma\) gives
\begin{eqnarray*}
\gamma\left(1+\frac{n}{2k^2}\right)
<
\left(
    1+\frac{2k-n}{2k^2+n}
\right)
\left(
    1+\frac{n}{2k^2}
\right)
=
\frac{k+1}{k}
=
\mathfrak c_{k}
\left(
    \frac{1}{(k-1)^2}
\right).
\end{eqnarray*}
We may therefore choose \(\varepsilon_0>0\) sufficiently small such that
\begin{eqnarray}\label{ineq-ck2}
    \mathfrak c_{k}
    \left(
        \frac{1}{(k-1)^2+\varepsilon_0}
    \right)
    >
    \gamma\left(1+\frac{n}{2k^2}\right)\geq \gamma\frac{P}{F}.
\end{eqnarray}

Apply Proposition~\ref{prop:shape-modulus} with
\(\varepsilon=\varepsilon_0\), and let
\(\eta_0=\eta(n,k,\varepsilon_0)>0\) be the corresponding constant. Set $\eta_{\mathrm{easy}}:=\frac{\eta_0}{2}$. By \eqref{ineq-ck2}, the coefficient of \(F\xi_1^2\) in \eqref{Qgamma-ck2} is
positive. Consequently,
\begin{eqnarray*}
    \mathcal Q_\gamma(\lambda;\xi)\geq0.
\end{eqnarray*}
\end{proof}

\section{G{\aa}rding Root Coordinates: the Hard Region}\label{hard}

Throughout this section, we retain the normalization and
assume that the largest eigenvalue is simple:
\begin{eqnarray*}
    a=\lambda_1=1>\lambda_2\geq\cdots\geq\lambda_n,
\end{eqnarray*}
and work in the hard region
\begin{eqnarray*}
    \frac{P}{F}
    >
    1+\frac{n}{2k^2}.
\end{eqnarray*}
In this region, the optimal concavity estimate used in
Section~\ref{easy} is no longer sufficient, since the
coefficient of \(F\xi_1^2\) in \eqref{Qgamma-ck2} does not remain
nonnegative. To handle this case, we use a more detailed representation of the \(\sigma_k\)-level sets. More precisely, we first represent the level set as a graph over the remaining variables and then introduce G{\aa}rding root coordinates. In these coordinates, the target
quadratic form separates into an explicit root-variable part and a
remainder with a favorable sign, reducing the hard-region estimate to the positivity of a simpler quadratic form.

\subsection{Tail Coordinate Representation}

Fix \(\lambda\in\Gamma_k\) satisfying
\(a=\lambda_1=1>\lambda_i\) for \(2\leq i\leq n\), and denote its
tail vector by
\begin{eqnarray*}
    \mu
    :=
    (\lambda_2,\ldots,\lambda_n)\in\mathbb R^{n-1}.
\end{eqnarray*}
For \(2\leq i\leq n\), define
\begin{eqnarray}\label{eq:tail-coordinates}
    z_i
    :=
    \frac{1}{1-\lambda_i}>0,
    \qquad
    \lambda_i=1-\frac{1}{z_i}.
\end{eqnarray}
Viewing \(\mu\) as a function of
\(z=(z_2,\ldots,z_n)\), we write
\begin{eqnarray*}
    \mu(z)
    :=
    \left(
        1-\frac{1}{z_2},
        \ldots,
        1-\frac{1}{z_n}
    \right),
\end{eqnarray*}
and set
\begin{eqnarray*}
    P(z)
    :=
    \sigma_{k-1}(\mu(z)),
    \qquad
    Q(z):=\sigma_k(\mu(z)).
\end{eqnarray*}
We have \(P(z)>0\), and
\begin{eqnarray}\label{eq:first-coordinate-affine}
    \sigma_k(s,\mu(z))
    =
    sP(z)+Q(z).
\end{eqnarray}

It is convenient to parametrize the level \(F\) logarithmically by
setting
\begin{eqnarray}\label{eq:inverse-level-x}
    x
    :=
    \log\left(
        \frac{(n-1)!}{k!(n-k)!}\frac{1}{F}
    \right).
\end{eqnarray}
The normalizing constant and the logarithmic parametrization are introduced only to simplify
the G{\aa}rding-root formulas below; neither is essential to the
argument. Define the graphing function
\begin{eqnarray}\label{eq:finite-level-graph}
    \psi(x,z)
    :=
    \frac{
        \dfrac{(n-1)!}{k!(n-k)!}e^{-x}-Q(z)
    }{
        P(z)
    }.
\end{eqnarray}
It follows from \eqref{eq:first-coordinate-affine} that
\begin{eqnarray}\label{eq:finite-level-relation}
    \sigma_k\bigl(\psi(x,z),\mu(z)\bigr)=
    \frac{(n-1)!}{k!(n-k)!}e^{-x}.
\end{eqnarray}
For the fixed vector \(\lambda\), the right-hand side equals
\(F=P(z)+Q(z)\), and therefore
\begin{eqnarray}\label{eq:psi-equals-one}
    \psi(x,z)=1.
\end{eqnarray}
We may assume
that \(\psi>0\) in a small neighborhood, and define
\begin{eqnarray}\label{eq:def-phi-from-psi}
    \phi(x,z)
    :=
    \log\psi(x,z).
\end{eqnarray}
Thus \(\phi=0\) at the point corresponding to \(\lambda\).

The purpose of introducing the \((x,z)\)-coordinates is to express the
original quadratic form in terms of the second derivatives of the
level-set graph (see Lemma \ref{lem:finite-level-graph-identity}). To obtain this representation, we relate a
variation in the original \(\lambda\)-variables to its components in
the new coordinates. The resulting Hessian identity will serve as the
starting point for the G{\aa}rding-root analysis.

Let
\begin{eqnarray*}
    \xi
    :=
    (\xi_1,\ldots,\xi_n)\in\mathbb R^n
\end{eqnarray*}
be arbitrary, and consider the corresponding variation
\(\lambda+t\xi\) of \(\lambda\). Its components in the
\((x,z)\)-coordinates are obtained by differentiating the coordinate
functions:
\begin{eqnarray}\label{eq:finite-level-direction}
    \dot x
    &:=&
    \left.
    \frac{\partial}{\partial t}
    \log\left(
        \frac{(n-1)!}{k!(n-k)!}
        \frac{1}{\sigma_k(\lambda+t\xi)}
    \right)
    \right|_{t=0}
    =
    -\frac{1}{F}
    \sum_{i=1}^nF^{ii}\xi_i,
    \\
    w_i
    &:=&
    \left.
    \frac{\partial}{\partial t}
    \frac{1}{1-(\lambda_i+t\xi_i)}
    \right|_{t=0}
    =
    z_i^2\xi_i,
    \qquad 2\leq i\leq n.
    \nonumber
\end{eqnarray}
We write
\begin{eqnarray*}
    w
    :=
    (w_2,\ldots,w_n),
    \qquad
    X:=(\dot x,w).
\end{eqnarray*}
With this notation, we obtain the following exact Hessian
representation of \(\mathcal Q_\gamma\).
\begin{lemma}\label{lem:finite-level-graph-identity}
At the point \((x,z)\) corresponding to the fixed vector \(\lambda\),
where \(\psi(x,z)=1\), let \(\xi\in\mathbb R^n\) be arbitrary,
and let \(X\) be its representation in the \((x,z)\)-coordinates, as
defined in \eqref{eq:finite-level-direction}. Then
\begin{eqnarray}\label{eq:phi-first-direction}
    D\psi(x,z)[X]
    =
    D\phi(x,z)[X]
    =
    \xi_1.
\end{eqnarray}
Moreover,
\begin{eqnarray}\label{eq:finite-level-graph-identity-psi}
    \frac{\mathcal Q_\gamma(\lambda;\xi)}{P}
    =
    D^2\psi(x,z)[X,X]
    -\gamma\{D\psi(x,z)[X]\}^2
    +\frac{F}{P}\dot x^{\,2}.
\end{eqnarray}
Equivalently, for \(\phi=\log\psi\),
\begin{eqnarray}\label{eq:finite-level-graph-identity}
    \frac{\mathcal Q_\gamma(\lambda;\xi)}{P}
    =
    D^2\phi(x,z)[X,X]
    -(\gamma-1)\{D\phi(x,z)[X]\}^2
    +\frac{F}{P}\dot x^{\,2}.
\end{eqnarray}
\end{lemma}

\begin{proof}
Consider the affine curve in the \((x,z)\)-coordinates given by
\begin{eqnarray*}
    x(t)
    :=
    x+t\dot x,
    \qquad
    z_i(t):=z_i+t w_i,
    \qquad 2\leq i\leq n,
\end{eqnarray*}
and set
\begin{eqnarray*}
    z(t)
    :=
    (z_2(t),\ldots,z_n(t)),\qquad
    \Lambda(t)
    :=
    \left(
        \psi(x(t),z(t)),
        1-\frac{1}{z_2(t)},
        \ldots,
        1-\frac{1}{z_n(t)}
    \right).
\end{eqnarray*}
Since \(\psi(x,z)=1\), we have \(\Lambda(0)=\lambda\). By Eq. \eqref{eq:finite-level-relation}, for all sufficiently small
\(t\),
\begin{eqnarray}\label{eq:finite-level-along-curve}
    \sigma_k(\Lambda(t))
    =
    \frac{(n-1)!}{k!(n-k)!}e^{-x(t)}.
\end{eqnarray}

For every \(i\geq2\),
\begin{eqnarray*}
    \left.
    \frac{\partial\Lambda_i(t)}{\partial t}
    \right|_{t=0}
    =
    \left.
    \frac{\partial}{\partial t}
    \left(
        1-\frac{1}{z_i+t w_i}
    \right)
    \right|_{t=0}
    =
    z_i^{-2}w_i
    =
    \xi_i.
\end{eqnarray*}
Hence the first derivative of the left-hand side of Eq. \eqref{eq:finite-level-along-curve} is
\begin{eqnarray*}
    \left.
    \frac{\partial}{\partial t}
    \sigma_k(\Lambda(t))
    \right|_{t=0}
    =
    P\,D\psi(x,z)[X]
    +
    \sum_{i=2}^nF^{ii}\xi_i.
\end{eqnarray*}
The first derivative of the right-hand side of Eq. \eqref{eq:finite-level-along-curve} is
\begin{eqnarray*}
    \left.
    \frac{\partial}{\partial t}
    \left(
        \frac{(n-1)!}{k!(n-k)!}e^{-x(t)}
    \right)
    \right|_{t=0}
    =
    -F\dot x
    =
    \sum_{i=1}^nF^{ii}\xi_i=
    P\xi_1+\sum_{i=2}^nF^{ii}\xi_i.
\end{eqnarray*}
Comparing the preceding identities and using \(P>0\), we obtain
\begin{eqnarray*}
    D\psi(x,z)[X]=\xi_1.
\end{eqnarray*}
Since \(\phi=\log\psi\) and \(\psi=1\), we have
\begin{eqnarray*}
    D\phi(x,z)[X]
    =
    \frac{D\psi(x,z)[X]}{\psi(x,z)}
    =
    \xi_1.
\end{eqnarray*}
This proves \eqref{eq:phi-first-direction}.

We next differentiate \eqref{eq:finite-level-along-curve} twice with
respect to \(t\). Since both \(x(t)\) and \(z(t)\) are affine,
\begin{eqnarray*}
    \left.
    \frac{\partial^2\Lambda_1(t)}{\partial t^2}
    \right|_{t=0}=D^2\psi(x,z)[X,X],
\end{eqnarray*}
and for every \(i\geq2\),
\begin{eqnarray*}
    \left.
    \frac{\partial^2\Lambda_i(t)}{\partial t^2}
    \right|_{t=0}
    =
    \left.
    \frac{\partial^2}{\partial t^2}
    \left(
        1-\frac{1}{z_i+t w_i}
    \right)
    \right|_{t=0}=
    -2z_i^{-3}w_i^2
    =
    -2\frac{\xi_i^2}{1-\lambda_i}.
\end{eqnarray*}
Thus, the second derivative of the left-hand side of Eq.
\eqref{eq:finite-level-along-curve} is
\begin{eqnarray*}
    \left.
    \frac{\partial^2}{\partial t^2}
    \sigma_k(\Lambda(t))
    \right|_{t=0}
    =
    \sum_{i,j=1}^nF^{ii,jj}\xi_i\xi_j
    +
    \sum_{i=1}^nF^{ii}
    \left.
    \frac{\partial^2\Lambda_i(t)}{\partial t^2}
    \right|_{t=0}.
\end{eqnarray*}
On the other hand, the second derivative of the right-hand side of Eq.
\eqref{eq:finite-level-along-curve} is
\begin{eqnarray*}
    \left.
    \frac{\partial^2}{\partial t^2}
    \left(
        \frac{(n-1)!}{k!(n-k)!}e^{-x(t)}
    \right)
    \right|_{t=0}
    =
    F\dot x^2.
\end{eqnarray*}
Comparing the above identities, we obtain
\begin{eqnarray*}
    F\dot x^2
    =
    \sum_{i,j=1}^nF^{ii,jj}\xi_i\xi_j
    +
    P\,D^2\psi(x,z)[X,X]
    -
    2\sum_{p=2}^n
    \frac{F^{pp}}{1-\lambda_p}\xi_p^2.
\end{eqnarray*}
Using
\begin{eqnarray*}
    -F\dot x
    =
    \sum_{i=1}^nF^{ii}\xi_i
\end{eqnarray*}
and rearranging, we obtain
\begin{eqnarray}\label{eq:finite-psi-second-identity}
    P\,D^2\psi(x,z)[X,X]
    =
    \frac{1}{F}
    \left(
        \sum_{i=1}^nF^{ii}\xi_i
    \right)^2
    -
    \sum_{i,j=1}^nF^{ii,jj}\xi_i\xi_j
    +
    2\sum_{p=2}^n
    \frac{F^{pp}}{1-\lambda_p}\xi_p^2.
\end{eqnarray}

Adding
\begin{eqnarray*}
    \frac{1}{F}
    \left(
        \sum_{i=1}^nF^{ii}\xi_i
    \right)^2
    -\gamma P\xi_1^2
\end{eqnarray*}
to both sides of \eqref{eq:finite-psi-second-identity}, using
\eqref{eq:phi-first-direction}, and dividing by \(P\), we obtain
\eqref{eq:finite-level-graph-identity-psi}.

Finally, since \(\phi=\log\psi\) and \(\psi=1\),
\begin{eqnarray*}
    D\phi[X]
    =
    D\psi[X],
    \qquad
    D^2\psi[X,X]
    =
    D^2\phi[X,X]+\{D\phi[X]\}^2.
\end{eqnarray*}
Substituting these identities into
\eqref{eq:finite-level-graph-identity-psi} proves
\eqref{eq:finite-level-graph-identity}.
\end{proof}

\subsection{G{\aa}rding Root Representation}

Let
\begin{eqnarray*}
    \mathbf 1=(1,\ldots,1)\in\mathbb R^{n-1},
    \qquad
    z^{-1}=(z_2^{-1},\ldots,z_n^{-1}).
\end{eqnarray*}
Since \(\sigma_k\), regarded as a polynomial on
\(\mathbb R^{n-1}\), is hyperbolic with respect to the direction
\(\mathbf 1\), we define the G{\aa}rding polynomial
associated with \(z^{-1}\) by
\begin{eqnarray}\label{eq:garding-polynomial}
    \mathcal G_z(t)
    :=
    \sigma_k(t\mathbf 1-z^{-1}).
\end{eqnarray}
This is a real-rooted polynomial of degree \(k\). Its roots
\begin{eqnarray*}
    \chi_1,\ldots,\chi_k
\end{eqnarray*}
are called the G{\aa}rding roots of \(z^{-1}\) with respect to
\(\sigma_k\) and the direction \(\mathbf 1\).

These roots can also be obtained from the following polynomial:
\begin{eqnarray*}
    p_z(t)
    :=
    \prod_{i=2}^n(t-z_i^{-1}).
\end{eqnarray*}
Taking the \((n-k-1)\)-st derivative, we have
\begin{eqnarray*}
    p_z^{(n-k-1)}(t)
    =
    (n-k-1)!\,
    \sigma_k(t\mathbf 1-z^{-1}).
\end{eqnarray*}
This identity is the key link in the root-based analysis below.
All the zeros of \(p_z\) are positive, so repeated applications of
Rolle's theorem give
\begin{eqnarray*}
    \chi_\alpha>0,
    \qquad
    1\leq\alpha\leq k.
\end{eqnarray*}

The hard-region condition gives
\(Q(z)=F-P(z)<0\). Since
\(\mathcal G_z(1)=Q(z)\), none of the roots \(\chi_\alpha\) equals
\(1\). We may therefore define the resolvent variables by
\begin{eqnarray}\label{eq:resolvent-coordinates}
    y_\alpha
    :=
    \frac{1}{1-\chi_\alpha}.
\end{eqnarray}
These variables are useful because both the root factorization of
\(\mathcal G_z\) and its logarithmic derivative at \(t=1\) take
particularly simple forms. More precisely, \(\mathcal G_z(1)\) can be expressed in terms of their
product, while
\(\mathcal G_z'(1)/\mathcal G_z(1)\) is given by their sum. This
simplification will be useful for rewriting the level-set relation and
characterizing the hard region.

When \(k=n-1\), these variables are closely related to the usual
Legendre transformation. Indeed, in this case,
\begin{eqnarray*}
    \mathcal G_z(t)
    =
    \prod_{i=2}^n(t-z_i^{-1}),
\end{eqnarray*}
so
\begin{eqnarray*}
    \chi_\alpha
    =
    z_{\alpha+1}^{-1}
    =
    1-\lambda_{\alpha+1},
    \qquad
    y_\alpha
    =
    \frac{1}{\lambda_{\alpha+1}}.
\end{eqnarray*}
Thus, in the positive-definite case, the variables \(y_\alpha\) are
exactly the eigenvalues of the inverse Hessian that appears under the
usual Legendre transformation:
\begin{eqnarray*}
    D^2u^*(Du(x))
    =
    \bigl(D^2u(x)\bigr)^{-1}.
\end{eqnarray*}
The variables introduced here may therefore be viewed as an extension
of these inverse-eigenvalue variables to the general
\(\sigma_k\) setting.

The value \(t=1\) is natural here because the largest eigenvalue has
been normalized to \(\lambda_1=1\). Moreover,
\begin{eqnarray*}
    \mathcal G_z(1)
    =
    \sigma_k(\mathbf 1-z^{-1})
    =
    \sigma_k(\mu)
    =
    Q(z).
\end{eqnarray*}
The root factorization gives
\begin{eqnarray}\label{eq:garding-root-factorization}
    \mathcal G_z(t)
    =
    \binom{n-1}{k}
    \prod_{\alpha=1}^k(t-\chi_\alpha).
\end{eqnarray}
Taking its logarithmic derivative at \(t=1\), we obtain
\begin{eqnarray}\label{eq:garding-log-derivative}
    \frac{\mathcal G_z'(1)}{\mathcal G_z(1)}
    =
    \sum_{\alpha=1}^k\frac{1}{1-\chi_\alpha}
    =
    \sum_{\alpha=1}^k y_\alpha.
\end{eqnarray}
Since
\(\mathcal G_z'(t)=(n-k)\sigma_{k-1}(t\mathbf 1-z^{-1})\), we have
\(\mathcal G_z'(1)=(n-k)P\). Combining this identity with
\eqref{eq:garding-root-factorization} and
\eqref{eq:garding-log-derivative} gives
\begin{eqnarray*}
    Q
    &=&
    \frac{(n-1)!}{k!(n-k)!}
    \frac{n-k}{\prod_{\alpha=1}^k y_\alpha},\qquad
    P=
    \frac{(n-1)!}{k!(n-k)!}
    \frac{\sum_{\alpha=1}^k y_\alpha}
         {\prod_{\alpha=1}^k y_\alpha},\\
    F&=&
    \frac{(n-1)!}{k!(n-k)!}
    \frac{\sum_{\alpha=1}^k y_\alpha+n-k}
         {\prod_{\alpha=1}^k y_\alpha},\qquad
    \frac{F}{P}=
    1+\frac{n-k}{\sum_{\alpha=1}^k y_\alpha}.
\end{eqnarray*}
Thus the variables \(y_\alpha\) express \(P\), \(Q\), and \(F\) in
terms of simple sums and products.

Writing \(y=(y_1,\ldots,y_k)\), we further obtain from
\begin{eqnarray*}
    e^{\phi(x,z)}
    =
    \frac{
        \dfrac{(n-1)!}{k!(n-k)!}e^{-x}-Q
    }{P}
\end{eqnarray*}
that
\begin{eqnarray}\label{eq:phi-root-formula}
    \phi(x,y)
    =
    \log\left|
        \prod_{\alpha=1}^k y_\alpha-(n-k)e^x
    \right|
    -
    \log\left|
        \sum_{\alpha=1}^k y_\alpha
    \right|
    -x.
\end{eqnarray}
At the point corresponding to \(\lambda\), we have \(\psi=1\).
Substituting this into \eqref{eq:finite-level-graph}, we obtain
\begin{eqnarray}\label{eq:finite-graph-relation}
    \prod_{\alpha=1}^k y_\alpha
    =
    e^x
    \left(
        \sum_{\alpha=1}^k y_\alpha+n-k
    \right).
\end{eqnarray}

We next characterize the hard region in root space. Since
\begin{eqnarray*}
    \mu=\mathbf 1-z^{-1}\in\Gamma_{k-1},
\end{eqnarray*}
we have
\begin{eqnarray*}
    \mathcal G_z(1+s)
    =
    \sigma_k(\mu+s\mathbf 1)
    =
    Q+
    \sum_{j=1}^k
    \binom{n-1-k+j}{j}
    \sigma_{k-j}(\mu)s^j.
\end{eqnarray*}
and the coefficient of \(s^j\) is
positive for every \(j\geq1\).

The hard-region condition gives \(Q=F-P<0\). Hence the constant term
of \(\mathcal G_z(1+s)\) is negative, while all its remaining
coefficients are positive. It follows that
\(\mathcal G_z(1+s)\) is strictly increasing for \(s>0\), and has exactly one positive zero. Returning to \(t=1+s\), exactly one
G{\aa}rding root satisfies \(\chi_\alpha>1\). Since all the
G{\aa}rding roots are positive and
\(\mathcal G_z(1)=Q\neq0\), all the remaining roots satisfy
\(0<\chi_\alpha<1\).

It follows that exactly one of the resolvent variables
\(y_\alpha\) is negative, while all the others satisfy
\(y_\alpha>1\). Thus every point in the hard region lies on the
\emph{one-negative resolvent branch} of the G{\aa}rding spectrum $\{y_\alpha\}$.

For simplicity, we now define
\begin{eqnarray}\label{eq:def-tau}
    \tau
    :=
    -\left(
        \sum_{\alpha=1}^k y_\alpha+n-k
    \right).
\end{eqnarray}
It follows that
\begin{eqnarray*}
    \frac{F}{P}
    =
    \frac{\tau}{\tau+n-k},
    \qquad
    \frac{P}{F}
    =
    1+\frac{n-k}{\tau}.
\end{eqnarray*}
Consequently, the hard-region condition
\begin{eqnarray*}
    \frac{P}{F}
    >
    1+\frac{n}{2k^2}
\end{eqnarray*}
is equivalent to
\begin{eqnarray}\label{eq:tau-range}
    0<\tau<\frac{2k^2(n-k)}{n}.
\end{eqnarray}

We now record some elementary consequences of the parameter range
that will be used in the hard-region estimates.

\begin{lemma}\label{lem:hard-parameter-signs}
Under the parameter assumptions of
Theorem~\ref{thm-crucial-ineq}, the following inequalities hold
in the hard region:
\begin{eqnarray*}
    (n-k)(2-\gamma)-(\gamma-1)\tau>0,\qquad (n-k)(2k-n\gamma)-n(\gamma-1)\tau>0.
\end{eqnarray*}
\end{lemma}

\begin{proof}
The left-hand sides of the two inequalities are affine functions
of \(\tau\). A direct calculation shows that both are positive at
\(\tau=0\) and \(\tau=2k^2(n-k)/n\). Hence they are positive
throughout the closed interval.
\end{proof}

After relabeling the resolvent variables, we may assume that the unique
negative component is the last one, and write
\begin{eqnarray}\label{eq:finite-one-negative-branch}
    y=(y_1,\ldots,y_{k-1},-\rho),
\end{eqnarray}
where \(\rho>0\) and \(y_j>1\) for \(1\leq j\leq k-1\). Set
\begin{eqnarray}\label{eq:def-R-rho}
    R:=\sum_{j=1}^{k-1}y_j.
\end{eqnarray}
Then
\(\rho=\tau+R+n-k\). The graph relation
\eqref{eq:finite-graph-relation} now becomes
\begin{eqnarray}\label{eq:finite-inverse-level}
    e^x
    =
    \frac{\rho}{\tau}
    \prod_{j=1}^{k-1}y_j.
\end{eqnarray}
Along any sequence in the hard region, the definition of \(x\) and
\eqref{eq:finite-inverse-level} show that \(F\to0\) if and only if
the right-hand side tends to \(+\infty\).

\subsection{Inverse G{\aa}rding Root Concavity and Quadratic Form Positivity}

On any region where the G{\aa}rding roots are distinct, define the
inverse roots by
\begin{eqnarray}\label{eq:inverse-roots}
    \zeta_l:=\frac{1}{\chi_l}>0,
    \qquad 1\leq l\leq k.
\end{eqnarray}
The main reason for introducing the inverse roots is that, when
arranged in increasing order, their partial sums are concave functions
of the tail coordinates \(z\). Combined with the corresponding ordering of the coefficients
\(\phi_{\zeta_{(l)}}\), this concavity allows us to apply summation
by parts and show that the remainder arising from the second-order
chain rule is nonnegative; see the last term in
\eqref{eq:finite-chain-decomposition}.

After substituting
\(y_l=\zeta_l/(\zeta_l-1)\) into
\eqref{eq:phi-root-formula}, we obtain a function of
\((x,\zeta)\), which, by a slight abuse of notation, we still
denote by \(\phi\). That is,
\begin{eqnarray*}
    \phi(x,z)=\phi\bigl(x,\zeta(z)\bigr).
\end{eqnarray*}

Let \(X=(\dot x,w)\) be arbitrary. Consider the affine
curve
\begin{eqnarray*}
    x(t):=x+t\dot x,
    \qquad
    z(t):=z+tw.
\end{eqnarray*}
Set
\begin{eqnarray*}
    \zeta_l(t):=\zeta_l(z(t)),\qquad \dot\zeta_l
    :=\left.\frac{d}{dt}\zeta_l(t)\right|_{t=0},
    \qquad
    \ddot\zeta_l
    :=\left.\frac{d^2}{dt^2}\zeta_l(t)\right|_{t=0}.
\end{eqnarray*}
At \(t=0\), the first- and second-order chain rules give
\begin{eqnarray}
    D_{(x,z)}\phi[X]
    &=&
    D_{(x,\zeta)}\phi[(\dot x,\dot\zeta)],
    \label{eq:first-chain-zeta}
    \\
    D_{(x,z)}^2\phi[X,X]
    &=&
    D_{(x,\zeta)}^2\phi
    [(\dot x,\dot\zeta),(\dot x,\dot\zeta)]
    +\sum_{l=1}^k\phi_{\zeta_l}\ddot\zeta_l.
    \label{eq:second-chain-zeta}
\end{eqnarray}
Consequently, Lemma~\ref{lem:finite-level-graph-identity} yields
\begin{eqnarray}
    \frac{\mathcal Q_\gamma(\lambda;\xi)}{P}=
    \left(
        D_{(x,\zeta)}^2\phi
        -(\gamma-1)
        D_{(x,\zeta)}\phi\otimes D_{(x,\zeta)}\phi
    \right)
    [(\dot x,\dot\zeta),(\dot x,\dot\zeta)]
    +\frac{\tau}{n-k+\tau}\dot x^{\,2}
    +\sum_{l=1}^k\phi_{\zeta_l}\ddot\zeta_l.
    \label{eq:finite-chain-decomposition}
\end{eqnarray}
The first two terms can be treated directly by completing the square. Lemma~\ref{lem:finite-composition} shows that the last term is nonnegative, using the concavity of the ordered partial sums of the inverse roots.

For simplicity, we make the invertible diagonal change of variables
\begin{eqnarray}\label{eq:root-direction-scaling}
    v_l
    :=\frac{(y_l-1)^2}{y_l}\dot\zeta_l.
\end{eqnarray}
Since
\begin{eqnarray*}
    y_l=\frac{\zeta_l}{\zeta_l-1},
    \qquad
    \frac{dy_l}{d\zeta_l}=-(y_l-1)^2,
\end{eqnarray*}
this definition is equivalent to
\begin{eqnarray*}
    \dot y_l=-y_l v_l.
\end{eqnarray*}

Since this diagonal change of variables is invertible, it suffices to
study the quadratic form in the variables \((\dot x,v)\). The following
lemma computes the first and second derivatives of \(\phi\) explicitly
in these variables and shows that all the diagonal coefficients are
positive except for the one associated with the unique negative
resolvent variable. This explicit representation allows us to complete
the square and analyze the remaining negative term.
\begin{lemma}\label{lem:finite-root-formula}
At the point corresponding to \(\lambda\), where \(\phi=0\), we have
\begin{eqnarray}
    D_{(x,\zeta)}\phi[(\dot x,\dot\zeta)]
    &=&
    -\frac{
        \tau\dot x+
        \sum_{l=1}^k(y_l+\tau)v_l
    }{n-k+\tau},
    \label{eq:Dphi-root}
    \\
    D_{(x,\zeta)}^2\phi
    [(\dot x,\dot\zeta),(\dot x,\dot\zeta)]
    &=&
    \sum_{l=1}^k d_l v_l^2
    +\frac{
        \left(\sum_{l=1}^k y_l v_l\right)^2
        +(n-k)\tau
        \left(\dot x+\sum_{l=1}^k v_l\right)^2
    }{(n-k+\tau)^2},
    \label{eq:D2phi-root}
\end{eqnarray}
where
\begin{eqnarray}\label{eq:def-dalpha}
    d_l
    :=\frac{2y_l^2+\tau(y_l+1)}
    {(n-k+\tau)(y_l-1)},
\end{eqnarray}
and $d_l>0$ for $1\leq l\leq k-1$; $d_k<0$.

Therefore, the first two terms of the right-hand side of Eq. \eqref{eq:finite-chain-decomposition} satisfy
\begin{eqnarray}
    &&
    \Bigl(
        D_{(x,\zeta)}^2\phi
        -(\gamma-1)
        D_{(x,\zeta)}\phi\otimes D_{(x,\zeta)}\phi
    \Bigr)
    [(\dot x,\dot\zeta),(\dot x,\dot\zeta)]
    +\frac{\tau}{n-k+\tau}\dot x^{\,2}
    \nonumber\\
    &=&
    \sum_{l=1}^k d_l v_l^2
    +\frac{\left(\sum_{l=1}^k y_l v_l\right)^2}
    {(n-k+\tau)^2}
    +\frac{(n-k)\tau}{(n-k+\tau)^2}
    \left(\dot x+\sum_{l=1}^k v_l\right)^2
    \nonumber\\
    &&
    -\frac{\gamma-1}{(n-k+\tau)^2}
    \left[
        \tau\dot x+\sum_{l=1}^k(y_l+\tau)v_l
    \right]^2
    +\frac{\tau}{n-k+\tau}\dot x^{\,2}.
    \label{eq:finite-root-quadratic}
\end{eqnarray}
\end{lemma}

\begin{proof}
To derive \eqref{eq:Dphi-root} and \eqref{eq:D2phi-root}, we
differentiate \eqref{eq:phi-root-formula} along the affine path
\begin{eqnarray*}
    (x(s),\zeta(s))
    :=
    (x+s\dot x,\zeta+s\dot\zeta)
\end{eqnarray*}
at \(s=0\). By \eqref{eq:root-direction-scaling}, the corresponding
derivatives of \(y_l\) are 
\begin{eqnarray}\label{eq:y-root-first-second}
    y_l'=-y_l v_l,
    \qquad
    y_l''=\frac{2y_l^2}{y_l-1}v_l^2.
\end{eqnarray}
At the point corresponding to \(\lambda\), we have
\begin{eqnarray*}
    \sum_{l=1}^k y_l=-(n-k+\tau),
    \qquad
    \prod_{l=1}^k y_l=-\tau e^x.
\end{eqnarray*}
Differentiating the three terms in
\eqref{eq:phi-root-formula} once gives
\begin{eqnarray*}
    \left.
    \frac{d}{ds}
    \log\left|
        \prod_{l=1}^k y_l(s)-(n-k)e^{x+s\dot x}
    \right|
    \right|_{s=0}
    &=&
    \frac{(n-k)\dot x-\tau\sum_{l=1}^k v_l}
    {n-k+\tau},
    \\
    -\left.
    \frac{d}{ds}
    \log\left|\sum_{l=1}^k y_l(s)\right|
    \right|_{s=0}=
    -\frac{\sum_{l=1}^k y_l v_l}{n-k+\tau}, &&\left.\frac{d}{ds}(-x-s\dot x)\right|_{s=0}
    =
    -\dot x.    
\end{eqnarray*}
Adding these together proves \eqref{eq:Dphi-root}.

Differentiating once more and using
\eqref{eq:y-root-first-second}, the first two logarithmic terms give
$$
    \left.
    \frac{d^2}{ds^2}
    \log\left|
        \prod_{l=1}^k y_l(s)-(n-k)e^{x+s\dot x}
    \right|
    \right|_{s=0}
    =
    \frac{(n-k)\tau}{(n-k+\tau)^2}
    \left(\dot x+\sum_{l=1}^k v_l\right)^2
    +\frac{\tau}{n-k+\tau}
    \sum_{l=1}^k\frac{y_l+1}{y_l-1}v_l^2
$$
and
\begin{eqnarray*}
    -\left.
    \frac{d^2}{ds^2}
    \log\left|\sum_{l=1}^k y_l(s)\right|
    \right|_{s=0}
    =
    \frac{\left(\sum_{l=1}^k y_l v_l\right)^2}
    {(n-k+\tau)^2}
    +\frac{2}{n-k+\tau}
    \sum_{l=1}^k\frac{y_l^2}{y_l-1}v_l^2.
\end{eqnarray*}
The second derivative of the last term \(-x-s\dot x\) vanishes.
Combining the coefficients of each \(v_l^2\) gives
\begin{eqnarray*}
    \frac{\tau(y_l+1)+2y_l^2}
    {(n-k+\tau)(y_l-1)}
    =d_l,
\end{eqnarray*}
which proves \eqref{eq:D2phi-root} and \eqref{eq:def-dalpha}.

Moreover, by direct calculation, we have $d_l>0$ for $1\leq l\leq k-1$ and $d_k<0$.

Finally, substituting \eqref{eq:Dphi-root} and
\eqref{eq:D2phi-root} into the root-coordinate quadratic form and
using
\begin{eqnarray*}
    \frac{F}{P}=\frac{\tau}{n-k+\tau}
\end{eqnarray*}
gives \eqref{eq:finite-root-quadratic}.
\end{proof}

Formula~\eqref{eq:finite-root-quadratic} shows that, apart from the
individual square terms
\begin{eqnarray*}
    \sum_{l=1}^k d_l v_l^2,
\end{eqnarray*}
the dependence of all the remaining terms on \(v\) occurs only through
the two linear combinations
\begin{eqnarray*}
    \sum_{l=1}^k y_l v_l,
    \qquad
    \sum_{l=1}^k v_l.
\end{eqnarray*}
In other words, if a root-coordinate variation
\(v=(v_1,\ldots,v_k)\) satisfies
\begin{eqnarray*}
    \sum_{l=1}^k y_l v_l=0,
    \qquad
    \sum_{l=1}^k v_l=0,
\end{eqnarray*}
then all the cross terms involving distinct \(v_l\) in
\eqref{eq:finite-root-quadratic} vanish, leaving only the diagonal
term \(\sum_{l=1}^k d_l v_l^2\) in the root variables.
Thus, regardless of the number \(k\) of roots, the coupled part always
lies in the at most two-dimensional subspace spanned by
\begin{eqnarray*}
    y=(y_1,\ldots,y_k)^{\mathsf T},
    \qquad
    \mathbf 1_k=(1,\ldots,1)^{\mathsf T}\in\mathbb R^k.
\end{eqnarray*}

Let
\begin{eqnarray*}
    D:=\operatorname{diag}(d_1,\ldots,d_k),
    \qquad
    U:=
    \begin{pmatrix}
        y_1 & y_2 & \cdots & y_{k-1} & -\rho\\
        1   & 1   & \cdots & 1       & 1
    \end{pmatrix}^{\mathsf T},
\end{eqnarray*}
and
\begin{eqnarray*}
    v:=(v_1,\ldots,v_k)^{\mathsf T},
    \qquad
    Y:=\sum_{l=1}^k y_l v_l,
    \qquad
    S:=\sum_{l=1}^k v_l.
\end{eqnarray*}
We first complete the square in \(\dot x\) on the right-hand side of
\eqref{eq:finite-root-quadratic}:
\begin{lemma}\label{lem:finite-normal-form}
For each fixed \(v\), the minimum of the right-hand side of
\eqref{eq:finite-root-quadratic} over \(\dot x\) is
\begin{eqnarray}\label{eq:root-block-after-x}
    v^{\mathsf T}
    \left[
        D+\frac{1}{(n-k+\tau)^2}
        U W_{\gamma}(\tau)U^{\mathsf T}
    \right]v,
\end{eqnarray}
where
\begin{eqnarray}
    W_{\gamma}(\tau)
    &=&
    \begin{pmatrix}
        2-\gamma
        &
        -(\gamma-1)\tau\\
        -(\gamma-1)\tau
        &
        (n-k)\tau-(\gamma-1)\tau^2
    \end{pmatrix}
    \nonumber\\
    &&
    -\frac{\tau}{2(n-k)+(2-\gamma)\tau}
    \binom{-(\gamma-1)}
          {n-k-(\gamma-1)\tau}
    \binom{-(\gamma-1)}
          {n-k-(\gamma-1)\tau}^{\mathsf T}.
    \label{eq:W-struct-new}
\end{eqnarray}
\end{lemma}

\begin{proof}
Using
\(U^{\mathsf T}v=(Y,S)^{\mathsf T}\), the right-hand side of
\eqref{eq:finite-root-quadratic} can be written as
\begin{eqnarray*}
    v^{\mathsf T}Dv
    +\frac{1}{(n-k+\tau)^2}
    \Bigl[
        Y^2+(n-k)\tau(\dot x+S)^2
        -(\gamma-1)(\tau\dot x+Y+\tau S)^2
    \Bigr]
    +\frac{\tau}{n-k+\tau}\dot x^{\,2}.
\end{eqnarray*}
Set
\begin{eqnarray*}
    c_\tau
    :=
    2(n-k)+(2-\gamma)\tau,
    \qquad
    L
    :=
    -(\gamma-1)Y
    +\{n-k-(\gamma-1)\tau\}S.
\end{eqnarray*}
By Lemma~\ref{lem:hard-parameter-signs},
\(c_\tau>0\). Expanding in \(\dot x\), completing the square, and using the definition of \(W_{\gamma}(\tau)\), we find that the preceding
expression equals
\begin{eqnarray*}
    \frac{\tau c_\tau}{(n-k+\tau)^2}
    \left(
        \dot x+\frac{L}{c_\tau}
    \right)^2
    +v^{\mathsf T}Dv
    +\frac{1}{(n-k+\tau)^2}
    \begin{pmatrix}Y&S\end{pmatrix}
    W_{\gamma}(\tau)
    \begin{pmatrix}Y\\S\end{pmatrix}.
\end{eqnarray*}
Since \(\tau>0\) and \(c_\tau>0\), the coefficient of the first
square is positive. Using again
\(U^{\mathsf T}v=(Y,S)^{\mathsf T}\), the remaining two terms become
exactly \eqref{eq:root-block-after-x}.
\end{proof}
To prove that \eqref{eq:root-block-after-x} is positive definite, we first study the positive definiteness of $W_{\gamma}(\tau)$:
\begin{proposition}\label{prop:finite-W-positive}
For \(0<\tau<2k^2(n-k)/n\), the matrix
\(W_{\gamma}(\tau)\) is positive definite.
\end{proposition}

\begin{proof}
A direct expansion of \eqref{eq:W-struct-new} gives
\begin{eqnarray*}
 (W_{\gamma})_{11}
 &=&
 \frac{
 2(n-k)(2-\gamma)
 +(3-2\gamma)\tau}
 {2(n-k)+(2-\gamma)\tau},
 \\
 (W_{\gamma})_{22}
 &=&
 \frac{
 \tau(n-k+\tau)
 \{n-k-(\gamma-1)\tau\}}
 {2(n-k)+(2-\gamma)\tau},
 \\
 \det W_{\gamma}
 &=&
 \frac{
 \tau(n-k+\tau)
 \{(n-k)(2-\gamma)-(\gamma-1)\tau\}}
 {2(n-k)+(2-\gamma)\tau}.
\end{eqnarray*}
By Lemma~\ref{lem:hard-parameter-signs},
\((W_{\gamma})_{11}>0\) and
\(\det W_{\gamma}>0\). Hence
\(W_{\gamma}(\tau)\) is positive definite.

For later use, we also compute
\begin{eqnarray}\label{eq:finite-W-inverse-11}
    \bigl(W_{\gamma}^{-1}\bigr)_{11}
    =
    \frac{(W_{\gamma})_{22}}
         {\det W_{\gamma}}
    =
    \frac{n-k-(\gamma-1)\tau}
    {(n-k)(2-\gamma)-(\gamma-1)\tau}.
\end{eqnarray}
\end{proof}
Let
\begin{eqnarray*}
    \mathcal M
    :=
    D+\frac{1}{(n-k+\tau)^2}
      U W_{\gamma}(\tau)U^{\mathsf T}
\end{eqnarray*}
be the matrix appearing in \eqref{eq:root-block-after-x}. Recall
\(\rho:=n-k+\tau+R\), and define
\begin{eqnarray}\label{eq:lower-pivot-explicit}
\underline{\mathcal P}(R,\tau)
&:=&
-\frac{2\rho^2-\tau\rho+\tau}
{(n-k+\tau)(\rho+1)}
\nonumber\\
&&
+\frac{\rho^2}{(n-k+\tau)^2}
\Biggl[
\frac{n-k-(\gamma-1)\tau}
{(n-k)(2-\gamma)-(\gamma-1)\tau}
\nonumber\\
&&\qquad
+
\frac{(R-k+1)(R-k+2)^2}
{(n-k+\tau)
 \{2(R-k+2)^2+\tau(R-k+3)\}}
\Biggr]^{-1}.
\end{eqnarray}

Although \(D\) has the single negative diagonal entry \(d_k\), the
term involving \(W_{\gamma}(\tau)\) is positive semidefinite. The
following proposition shows that, whenever
\(\underline{\mathcal P}(R,\tau)>0\), this positive term controls the
negative part from \(D\), and consequently \(\mathcal M\) is
positive definite.

\begin{proposition}\label{prop:negative-pivot-lower-bound}
If \(\underline{\mathcal P}(R,\tau)>0\), then \(\mathcal M\) is
positive definite.
\end{proposition}

\begin{proof}
Let \(v=(v_1,\ldots,v_k)^{\mathsf T}\) be arbitrary, and set
\begin{eqnarray*}
    e_1:=(1,0)^{\mathsf T},
    \qquad
    q:=U^{\mathsf T}v=(Y,S)^{\mathsf T},
    \qquad
    Y=\sum_{j=1}^{k-1}y_jv_j-\rho v_k.
\end{eqnarray*}
Since \(W_{\gamma}(\tau)>0\), the Cauchy--Schwarz inequality gives
\begin{eqnarray*}
    Y^2
    =
    (e_1^{\mathsf T}q)^2
    \leq
    \bigl(e_1^{\mathsf T}W_{\gamma}(\tau)^{-1}e_1\bigr)
    \bigl(q^{\mathsf T}W_{\gamma}(\tau)q\bigr).
\end{eqnarray*}
Consequently,
\begin{eqnarray}\label{vTMv}
    v^{\mathsf T}\mathcal Mv
    &\geq&
    \sum_{j=1}^{k-1}d_jv_j^2+d_kv_k^2
    +
    \frac{Y^2}{
        (n-k+\tau)^2
        \bigl(W_{\gamma}(\tau)^{-1}\bigr)_{11}
    }.
\end{eqnarray}

Since \(\rho>0\), the change of coordinates
\begin{eqnarray*}
    (v_1,\ldots,v_k)
    \longmapsto
    (v_1,\ldots,v_{k-1},Y)
\end{eqnarray*}
is invertible, with
\begin{eqnarray*}
v_k=\frac{\sum_{j=1}^{k-1}y_jv_j-Y}{\rho}.
\end{eqnarray*}
Thus, in the coordinates
\((v_1,\ldots,v_{k-1},Y)\), the right-hand side of \eqref{vTMv} becomes
\begin{eqnarray*}
    \sum_{j=1}^{k-1}d_jv_j^2
    +
    \frac{Y^2}{
        (n-k+\tau)^2
        \bigl(W_{\gamma}(\tau)^{-1}\bigr)_{11}
    }
    -
    \frac{-d_k}{\rho^2}
    \left(
        \sum_{j=1}^{k-1}y_jv_j-Y
    \right)^2.
\end{eqnarray*}
Here \(d_j>0\) for \(1\leq j\leq k-1\), \(d_k<0\), and
\(\bigl(W_{\gamma}(\tau)^{-1}\bigr)_{11}>0\).
Applying \cite[Lemma~3.2]{QiuZhou2024} with
\begin{eqnarray*}
    a_j^2=d_j,
    \qquad
    b_j=\frac{\sqrt{-d_k}}{\rho}y_j
    \quad(1\leq j\leq k-1),
\end{eqnarray*}
and
\begin{eqnarray*}
    a_k^2
    =
    \frac{1}{
        (n-k+\tau)^2
        \bigl(W_{\gamma}(\tau)^{-1}\bigr)_{11}
    },
    \qquad
    b_k=-\frac{\sqrt{-d_k}}{\rho},
\end{eqnarray*}
we see that this quadratic form is positive definite provided that
\begin{eqnarray*}
    1+
    \frac{d_k}{\rho^2}
    \left[
        (n-k+\tau)^2
        \left(W_{\gamma}(\tau)^{-1}\right)_{11}
        +
        \sum_{j=1}^{k-1}\frac{y_j^2}{d_j}
    \right]
    >0.
\end{eqnarray*}
Equivalently,
\begin{eqnarray*}
    d_k+
    \frac{\rho^2}{
        (n-k+\tau)^2
        \left(W_{\gamma}(\tau)^{-1}\right)_{11}
        +
        \displaystyle\sum_{j=1}^{k-1}\frac{y_j^2}{d_j}
    }
    >0.
\end{eqnarray*}
By \eqref{eq:def-dalpha}, for \(1\leq j\leq k-1\),
\begin{eqnarray*}
    \frac{y_j^2}{d_j}
    =
    (n-k+\tau)(y_j-1)
    \frac{y_j^2}{2y_j^2+\tau(y_j+1)}.
\end{eqnarray*}
The function
\begin{eqnarray*}
    f(s):=
    \frac{s^2}{2s^2+\tau(s+1)}
\end{eqnarray*}
is increasing for \(s\geq1\). Since
\begin{eqnarray*}
    y_j\leq R-k+2,
    \qquad
    \sum_{j=1}^{k-1}(y_j-1)=R-k+1,
\end{eqnarray*}
it follows that
\begin{eqnarray*}
    \sum_{j=1}^{k-1}\frac{y_j^2}{d_j}
    &=&
    (n-k+\tau)
    \sum_{j=1}^{k-1}(y_j-1)f(y_j)
    \\
    &\leq&
    (n-k+\tau)f(R-k+2)
    \sum_{j=1}^{k-1}(y_j-1)
    \\
    &=&
    (n-k+\tau)(R-k+1)
    \frac{(R-k+2)^2}
    {2(R-k+2)^2+\tau(R-k+3)}.
\end{eqnarray*}
Together with \eqref{eq:def-dalpha},
\eqref{eq:finite-W-inverse-11}, and the definition of
\(\underline{\mathcal P}(R,\tau)\), a direct calculation yields
\begin{eqnarray*}
    d_k+
    \frac{\rho^2}{
        (n-k+\tau)^2
        \bigl(W_{\gamma}(\tau)^{-1}\bigr)_{11}
        +
        \displaystyle\sum_{j=1}^{k-1}\frac{y_j^2}{d_j}
    }
    \geq
    \underline{\mathcal P}(R,\tau).
\end{eqnarray*}
Hence by \cite[Lemma~3.2]{QiuZhou2024}, \(\mathcal M\) is positive definite provided that $\underline{\mathcal P}(R,\tau)>0$.
\end{proof}
We are now ready to prove that \(\mathcal M\), and hence the matrix in \eqref{eq:root-block-after-x}, is positive definite.
\begin{lemma}\label{lem:root-block-positive}
Under the parameter assumptions of
Theorem~\ref{thm-crucial-ineq}, there exists
\begin{eqnarray*}
    x_0=x_0(n,k,\gamma)<\infty
\end{eqnarray*}
such that, throughout the hard region, the matrix in
\eqref{eq:root-block-after-x} is positive definite whenever
\(x\geq x_0\).
\end{lemma}

\begin{proof}
Proposition~\ref{prop:negative-pivot-lower-bound} reduces the desired
positive definiteness to proving that
\(\underline{\mathcal P}(R,\tau)>0\). By
\eqref{eq:finite-inverse-level}, the variable \(x\) can tend to
\(+\infty\) only when \(\tau\) approaches \(0\) or \(R\) becomes
unbounded. It therefore suffices to establish uniform positivity near the two boundary \(\tau=0\) and \(R=\infty\).

We compactify the interval \(R\in[k-1,\infty)\) by setting
\begin{eqnarray*}
    s=\frac{1}{1+R}.
\end{eqnarray*}
After substituting \(R=s^{-1}-1\) into
\eqref{eq:lower-pivot-explicit}, we see that
\(\underline{\mathcal P}\) extends continuously to
\begin{eqnarray*}
    0\leq s\leq\frac{1}{k},
    \qquad
    0\leq\tau\leq\frac{2k^2(n-k)}{n},
\end{eqnarray*}
where \(s=0\) corresponds to \(R=\infty\). Its values on the two
boundary faces are
\begin{eqnarray}
 \underline{\mathcal P}(R,0)
 &=&
 \frac{
 2(n-k+R)^2(2k-n\gamma)}
 {(n-k)(n-k+1+R)
  \{2(n-k)+(2-\gamma)(R-k+1)\}},
 \label{eq:lower-pivot-tau-zero}
 \\
 \underline{\mathcal P}(\infty,\tau)
 &=&
 \frac{
 2\{(n-k)(2k-n\gamma)-n(\gamma-1)\tau\}}
 {(n-k+\tau)
  \{(n-k)(2-\gamma)-(\gamma-1)\tau\}}.
 \label{eq:lower-pivot-R-infinity}
\end{eqnarray}
At their intersection \((R,\tau)=(\infty,0)\), both expressions take
the value
\begin{eqnarray*}
    \frac{2(2k-n\gamma)}
    {(n-k)(2-\gamma)}>0.
\end{eqnarray*}

The parameter assumption \(\gamma<2k/n\) shows that
\eqref{eq:lower-pivot-tau-zero} is strictly positive. The numerator
and denominator in \eqref{eq:lower-pivot-R-infinity} are positive by
Lemma~\ref{lem:hard-parameter-signs}. Hence, by continuity and
compactness, there exist constants
\begin{eqnarray*}
    0<\tau_0<\frac{2k^2(n-k)}{n},
    \qquad
    R_0>k-1,
\end{eqnarray*}
such that
\begin{eqnarray}\label{eq:positive-boundary-neighborhoods}
    \underline{\mathcal P}(R,\tau)>0
\end{eqnarray}
whenever
\begin{eqnarray*}
    0\leq\tau\leq\tau_0
    \qquad\text{or}\qquad
    R\geq R_0.
\end{eqnarray*}
Consequently, there exists a constant
\(x_0=x_0(n,k,\gamma)<\infty\) such that, throughout the hard region,
\begin{eqnarray*}
    x\geq x_0
    \quad\Longrightarrow\quad
    \underline{\mathcal P}(R,\tau)>0.
\end{eqnarray*}
Proposition~\ref{prop:negative-pivot-lower-bound} then shows that the
matrix in \eqref{eq:root-block-after-x} is positive definite
whenever \(x\geq x_0\). This proves
Lemma~\ref{lem:root-block-positive}.

\end{proof}

Back to \eqref{eq:finite-chain-decomposition}. It remains
to control the second-order chain-rule remainder
\begin{eqnarray*}
    \sum_{l=1}^k
    \phi_{\zeta_l}\ddot\zeta_l.
\end{eqnarray*}
Its nonnegativity ultimately follows from the concavity of the partial sums of the ordered inverse roots. The proof of this concavity is based on an induction argument. We begin with the following induction lemma. Recall that
\begin{eqnarray*}
    p_z(\chi):=\prod_{i=2}^n(\chi-z_i^{-1}).
\end{eqnarray*}
For \(0\leq r\leq n-k-1\), write
\begin{eqnarray*}
    p_z^{(r)}(\chi)
    :=\frac{d^r}{d\chi^r}p_z(\chi),
    \qquad
    m_r:=n-1-r.
\end{eqnarray*}
By Rolle's theorem, the \(m_r\) zeros of
\(p_z^{(r)}\) are all positive. Denote their reciprocals by
\begin{eqnarray}\label{eq:stage-reciprocal-roots}
    x_1^{(r)}(z)\leq\cdots\leq x_{m_r}^{(r)}(z),
\end{eqnarray}
and define the partial sums
\begin{eqnarray}\label{eq:stage-prefix-sums}
    L_j^{(r)}(z)
    :=\sum_{i=1}^j x_i^{(r)}(z),
    \qquad 1\leq j\leq m_r.
\end{eqnarray}
In the inductive step, fix
\begin{eqnarray*}
    0\leq r<n-k-1,
    \qquad
    m:=m_r=n-1-r.
\end{eqnarray*}
For brevity, write
\begin{eqnarray*}
    x_i(z):=x_i^{(r)}(z),
    \quad 1\leq i\leq m,\qquad\eta_i(z):=x_i^{(r+1)}(z),
    \quad 1\leq i\leq m-1.
\end{eqnarray*}
Thus
\begin{eqnarray*}
    x_1(z)\leq\cdots\leq x_m(z),
    \qquad
    \eta_1(z)\leq\cdots\leq\eta_{m-1}(z).
\end{eqnarray*}

\begin{lemma}\label{lem:one-derivative-preserves-prefix-concavity}
For \(1\leq l\leq m\), set
\begin{eqnarray*}
    L_l(z):=L_l^{(r)}(z)
    =\sum_{i=1}^{l}x_i(z).
\end{eqnarray*}
Suppose that every \(L_l\) is concave in \(z\). Then, for each
\(1\leq j\leq m-1\), the function
\begin{eqnarray*}
    z\longmapsto\sum_{i=1}^j\eta_i(z)
    =L_j^{(r+1)}(z)
\end{eqnarray*}
is also concave.
\end{lemma}

\begin{proof}
The main idea is to express each partial sum of the new inverse roots as
the infimum of nonnegative linear combinations of the old partial sums.
We divide the proof into three steps.

\medskip
\noindent
\textbf{Step 1: A matrix representation of the new inverse roots.}

Fix \(z\), set
\begin{eqnarray*}
    X:=\operatorname{diag}(x_1,\ldots,x_m),
    \qquad
    e:=\frac{1}{\sqrt m}(1,\ldots,1)^{\mathsf T}.
\end{eqnarray*}
Choose a constant matrix \(V\in\mathbb R^{m\times(m-1)}\) whose columns form an orthonormal basis of \(e^\perp\). That is,
\begin{eqnarray*}
    V^{\mathsf T}V=I_{m-1},
    \qquad
    VV^{\mathsf T}=I_m-ee^{\mathsf T}.
\end{eqnarray*}

We now show how the characteristic polynomial of the
compressed matrix \(V^{\mathsf T}X^{-1}V\) is related to
\(p_z^{(r+1)}\). Set
\begin{eqnarray*}
    A(\chi):=\chi I_m-X^{-1},
    \qquad
    Q:=(V,e).
\end{eqnarray*}
We have
\begin{eqnarray*}
    Q^{\mathsf T}A(\chi)Q
    =
    \begin{pmatrix}
        V^{\mathsf T}A(\chi)V
        &
        V^{\mathsf T}A(\chi)e
        \\
        e^{\mathsf T}A(\chi)V
        &
        e^{\mathsf T}A(\chi)e
    \end{pmatrix},
\end{eqnarray*}
and
\begin{eqnarray*}
    V^{\mathsf T}A(\chi)V
    =\chi I_{m-1}-V^{\mathsf T}X^{-1}V.
\end{eqnarray*}

By the definition of the adjugate, the \((m,m)\)-entry of
\(\operatorname{adj}(Q^{\mathsf T}A(\chi)Q)\) is the determinant of
the submatrix obtained by deleting the last row and the last column of
\(Q^{\mathsf T}A(\chi)Q\). Hence
\begin{eqnarray*}
    \det\!\left(
        \chi I_{m-1}-V^{\mathsf T}X^{-1}V
    \right)
    =
    \left[
        \operatorname{adj}
        \bigl(Q^{\mathsf T}A(\chi)Q\bigr)
    \right]_{mm}.
\end{eqnarray*}
The adjugate satisfies
\begin{eqnarray*}
    \operatorname{adj}(Q^{\mathsf T}AQ)
    =Q^{\mathsf T}\operatorname{adj}(A)Q
\end{eqnarray*}
under orthogonal similarity. The last column of \(Q\) is \(e\), and therefore
\begin{eqnarray}\label{eq:compressed-characteristic-polynomial}
    \det\!\left(
        \chi I_{m-1}-V^{\mathsf T}X^{-1}V
    \right)
    =e^{\mathsf T}
    \operatorname{adj}(\chi I_m-X^{-1})e.
\end{eqnarray}
A direct calculation yields
\begin{eqnarray*}
    e^{\mathsf T}
    \operatorname{adj}(\chi I_m-X^{-1})e
    =
    \frac{1}{m}
    \sum_{i=1}^m
    \prod_{l\ne i}(\chi-x_l^{-1})
    =
    \frac{(m-1)!}{(n-1)!}\,p_z^{(r+1)}(\chi).
\end{eqnarray*}
Substituting this into
\eqref{eq:compressed-characteristic-polynomial}, we obtain
\begin{eqnarray}\label{eq:derivative-compression-identity}
    \det\!\left(
        \chi I_{m-1}-V^{\mathsf T}X^{-1}V
    \right)
    =\frac{(m-1)!}{(n-1)!}\,p_z^{(r+1)}(\chi).
\end{eqnarray}
Thus the zeros of \(p_z^{(r+1)}\) are precisely the eigenvalues of
\(V^{\mathsf T}X^{-1}V\). Consequently, their reciprocals
\(\eta_1,\ldots,\eta_{m-1}\) are precisely the eigenvalues of the
positive-definite matrix
\begin{eqnarray*}
    \mathcal H
    :=\left(V^{\mathsf T}X^{-1}V\right)^{-1}.
\end{eqnarray*}

\medskip
\noindent
\textbf{Step 2: A variational formula for the partial sums.}

We \textbf{claim} that for every \(u\in\mathbb R^{m-1}\),
\begin{eqnarray}\label{eq:compressed-energy}
    u^{\mathsf T}\mathcal H u
    =\min_{V^{\mathsf T}\omega=u}
    \omega^{\mathsf T}X\omega.
\end{eqnarray}
Indeed, define
\begin{eqnarray*}
    \omega_*
    :=X^{-1}V\mathcal H u.
\end{eqnarray*}
Since \(\mathcal H^{-1}=V^{\mathsf T}X^{-1}V\), we have
\begin{eqnarray*}
    V^{\mathsf T}\omega_*
    =V^{\mathsf T}X^{-1}V\mathcal H u
    =u,
\end{eqnarray*}
so \(\omega_*\) satisfies the constraint.

Let \(\omega\) be any other vector satisfying
\(V^{\mathsf T}\omega=u\), and set \(h:=\omega-\omega_*\). Then
\(V^{\mathsf T}h=0\), and hence
\begin{eqnarray*}
    h^{\mathsf T}X\omega_*
    =h^{\mathsf T}V\mathcal H u
    =(V^{\mathsf T}h)^{\mathsf T}\mathcal H u
    =0.
\end{eqnarray*}
It follows that
\begin{eqnarray*}
    \omega^{\mathsf T}X\omega
    =
    (\omega_*+h)^{\mathsf T}X(\omega_*+h)
    =
    \omega_*^{\mathsf T}X\omega_*+h^{\mathsf T}Xh
    \geq
    \omega_*^{\mathsf T}X\omega_*.
\end{eqnarray*}
As \(X\) is positive definite, equality holds only when \(h=0\), so
\(\omega_*\) is the unique minimizer. Finally,
\begin{eqnarray*}
    \omega_*^{\mathsf T}X\omega_*=
    \omega_*^{\mathsf T}V\mathcal H u=
    (V^{\mathsf T}\omega_*)^{\mathsf T}\mathcal H u=
    u^{\mathsf T}\mathcal H u,
\end{eqnarray*}
which proves \eqref{eq:compressed-energy}.

Recall that
\(\eta_1\leq\cdots\leq\eta_{m-1}\) are the eigenvalues of the real
symmetric matrix \(\mathcal H\). The standard variational
characterization of the sum of its \(j\) smallest eigenvalues gives
\begin{eqnarray}\label{eq:eigenvalue-sum-minimum}
    \sum_{i=1}^j\eta_i(z)
    =
    \min_{\substack{
        U\in\mathbb R^{(m-1)\times j}\\
        U^{\mathsf T}U=I_j
    }}
    \operatorname{tr}
    \bigl(U^{\mathsf T}\mathcal H(z)U\bigr).
\end{eqnarray}
For a fixed \(U\), applying \eqref{eq:compressed-energy} to each
column of \(U\) yields
\begin{eqnarray*}
    \operatorname{tr}
    \bigl(U^{\mathsf T}\mathcal H(z)U\bigr)
    =
    \min_{V^{\mathsf T}W=U}
    \operatorname{tr}
    \bigl(W^{\mathsf T}X(z)W\bigr).
\end{eqnarray*}
Substituting this into \eqref{eq:eigenvalue-sum-minimum} and eliminating
the intermediate variable \(U=V^{\mathsf T}W\), we obtain
\begin{eqnarray}\label{eq:prefix-variational}
    \sum_{i=1}^j\eta_i(z)=\min_{\substack{
        W\in\mathbb R^{m\times j}\\
        W^{\mathsf T}VV^{\mathsf T}W=I_j
    }}
    \operatorname{tr}
    \bigl(W^{\mathsf T}X(z)W\bigr)=
    \min_{\substack{
        W\in\mathbb R^{m\times j}\\
        W^{\mathsf T}(I_m-ee^{\mathsf T})W=I_j
    }}
    \sum_{i=1}^m
    x_i(z)\lvert W_{i\bullet}\rvert^2,
\end{eqnarray}
where \(W_{i\bullet}\) denotes the \(i\)-th row of \(W\). In particular, the class of admissible matrices in
\eqref{eq:prefix-variational} depends only on \(m\) and \(j\) and is
independent of \(z\).

\medskip
\noindent
\textbf{Step 3: Row rearrangement and Abel summation.}

For a fixed admissible matrix \(W\) in
\eqref{eq:prefix-variational}, set
\begin{eqnarray*}
    c_i:=\lvert W_{i\bullet}\rvert^2.
\end{eqnarray*}
Every permutation matrix \(P\) satisfies \(Pe=e\), so replacing \(W\)
by \(PW\) preserves the constraint. Since
\(x_1(z)\leq\cdots\leq x_m(z)\), repeated pairwise row exchanges show
that the minimum in \eqref{eq:prefix-variational} may be restricted to
admissible matrices satisfying
\begin{eqnarray*}
    c_1\geq c_2\geq\cdots\geq c_m\geq0.
\end{eqnarray*}
This restricted class of admissible matrices is independent of \(z\).

Abel summation gives
\begin{eqnarray}\label{eq:abel-prefix-combination}
    \sum_{i=1}^m c_i x_i(z)
    =
    c_mL_m(z)
    +
    \sum_{l=1}^{m-1}
    (c_l-c_{l+1})L_l(z).
\end{eqnarray}
All the coefficients on the right-hand side are nonnegative constants
independent of \(z\). Hence, for every fixed \(W\) in the restricted
class, the right-hand side of
\eqref{eq:abel-prefix-combination} is concave in \(z\). Taking the
pointwise minimum over this fixed class in
\eqref{eq:prefix-variational}, we conclude that
\begin{eqnarray*}
    z\longmapsto
    \sum_{i=1}^j\eta_i(z)
    =
    L_j^{(r+1)}(z)
\end{eqnarray*}
is concave.
\end{proof}

Using the preceding induction lemma, we now prove the concavity of the partial sums.
\begin{lemma}\label{lem:prefix-concavity-input}
On the region where the G{\aa}rding roots are distinct, relabel the
inverse roots in increasing order as
\begin{eqnarray*}
    \zeta_{(1)}<\cdots<\zeta_{(k)}.
\end{eqnarray*}
Let \(z=z(t)\) be an affine line segment lying entirely in the region where the G{\aa}rding roots are distinct. Then, for every \(1\leq j\leq k\), we have
\begin{eqnarray}\label{eq:prefix-concavity-input}
    \frac{d^2}{dt^2}
    \sum_{l=1}^j
    \zeta_{(l)}(z(t))
    \leq0.
\end{eqnarray}
\end{lemma}

\begin{proof}
Since \(z(t)=z+tw\) is affine in \(t\), it suffices to show that
\begin{eqnarray*}
    L_j^{(n-k-1)}(z)
    =
    \sum_{l=1}^j\zeta_{(l)}(z)
\end{eqnarray*}
is concave in \(z\).

By repeated application of Lemma~\ref{lem:one-derivative-preserves-prefix-concavity}, it remains only to verify the base case \(r=0\). In that case, the numbers \(x_i^{(0)}(z)\) are the increasing
rearrangement of \(z_2,\ldots,z_n\). Therefore,
\begin{eqnarray}\label{eq:initial-prefix-minimum}
    L_j^{(0)}(z)
    =\min_{\substack{I\subset\{2,\ldots,n\}\\ |I|=j}}
      \sum_{i\in I}z_i.
\end{eqnarray}
This is the pointwise minimum of finitely many linear functions and is
therefore concave. This proves \eqref{eq:prefix-concavity-input}.
\end{proof}

It follows from \eqref{eq:Dphi-root} that
\begin{eqnarray}\label{eq:phi-zeta-explicit}
    \phi_{\zeta_l}
    =
    -\frac{
        (y_l-1)^2(y_l+\tau)
    }{
        y_l(n-k+\tau)
    }
    <0.
\end{eqnarray}
Indeed, for \(1\leq l\leq k-1\), the sign follows immediately from
\(y_l>1\). For the unique negative resolvent coordinate
\(y_k=-\rho\), we have
\begin{eqnarray*}
    y_k+\tau=-(n-k+R)<0.
\end{eqnarray*}
Together with \(y_k<0\), this again gives
\(\phi_{\zeta_k}<0\).

For \(i\neq j\), direct subtraction gives
\begin{eqnarray}\label{eq:finite-gradient-divided-difference}
    (n-k+\tau)
    \frac{
        \phi_{\zeta_i}-\phi_{\zeta_j}
    }{
        y_i-y_j
    }
    =
    2-y_i-y_j-\tau+\frac{\tau}{y_i y_j}.
\end{eqnarray}
If \(y_i,y_j>1\), the right-hand side is strictly negative. If
\(y_i=y_k=-\rho<0\) and \(1\leq j\leq k-1\), it is equal to
\begin{eqnarray*}
    n-k+R+2-y_j-\frac{\tau}{\rho y_j}>0,
\end{eqnarray*}
since \(R-y_j\geq0\) and
\(\tau/(\rho y_j)<1\).

The unique negative resolvent coordinate corresponds to
\(\zeta_k\in(0,1)\), while all the other inverse roots are
greater than \(1\). Moreover, \(y\) is strictly decreasing as a
function of \(\zeta\). Therefore,
\eqref{eq:finite-gradient-divided-difference} gives
\begin{eqnarray}\label{eq:finite-gradient-order}
    \phi_{\zeta_{(1)}}
    \leq\cdots\leq
    \phi_{\zeta_{(k)}}<0.
\end{eqnarray}

We now combine the partial-sum concavity in
Lemma~\ref{lem:prefix-concavity-input} with the coefficient ordering
in \eqref{eq:finite-gradient-order} to prove that the chain-rule
remainder is nonnegative.
\begin{lemma}\label{lem:finite-composition}
Let \(z=z(t)\) be an affine line segment contained in the region of
distinct roots. Then
\begin{eqnarray}\label{eq:finite-composition-remainder}
    \sum_{l=1}^k
    \phi_{\zeta_l}\zeta_l''(t)
    \geq0.
\end{eqnarray}
\end{lemma}

\begin{proof}
Set
\begin{eqnarray*}
    B_j(t)
    :=
    \sum_{l=1}^j
    \zeta_{(l)}(z(t)),
    \qquad
    B_0:=0.
\end{eqnarray*}
By Lemma~\ref{lem:prefix-concavity-input}, \(B_j''\leq0\). Moreover,
\begin{eqnarray*}
    \zeta_{(l)}''
    =
    B_l''-B_{l-1}''.
\end{eqnarray*}
It follows that
\begin{eqnarray*}
    \sum_{l=1}^k
    \phi_{\zeta_{(l)}}\zeta_{(l)}''
    &=&
    \phi_{\zeta_{(k)}}B_k''
    +
    \sum_{j=1}^{k-1}
    \left(
        \phi_{\zeta_{(j)}}
        -
        \phi_{\zeta_{(j+1)}}
    \right)B_j''
    \geq0,
\end{eqnarray*}
where we have used \(B_j''\leq0\) and
\eqref{eq:finite-gradient-order}. This proves
\eqref{eq:finite-composition-remainder}.
\end{proof}

We are now ready to complete the proof of
Theorem~\ref{thm-crucial-ineq}.

\begin{proof}[Proof of Theorem~\ref{thm-crucial-ineq}]
Let \(x_0\) be given by
Lemma~\ref{lem:root-block-positive}, and set
\[
    F_{\mathrm{hard}}
    :=
    \frac{(n-1)!}{k!(n-k)!}e^{-(x_0+1)},
    \qquad
    \eta_*
    :=
    \min\{\eta_{\mathrm{easy}},F_{\mathrm{hard}}\}.
\]
Suppose first that \(a=\lambda_1=1\). In the hard region,
\(F\leq F_{\mathrm{hard}}\) implies \(x\geq x_0+1\) by
\eqref{eq:inverse-level-x}. When the G{\aa}rding roots are distinct,
Lemmas~\ref{lem:finite-normal-form},
\ref{lem:root-block-positive}, and
\ref{lem:finite-composition}, together with
\eqref{eq:finite-chain-decomposition}, give
\(\mathcal Q_\gamma(\lambda;\xi)/P\geq0\), and hence
\(\mathcal Q_\gamma(\lambda;\xi)\geq0\).

The same conclusion holds when some G{\aa}rding roots coincide.
Indeed, approximate \(\lambda\) by normalized admissible vectors
\(\lambda^{(m)}\) in the hard region whose G{\aa}rding roots are
distinct. The corresponding \(x_m\) converge to \(x\), so
\(x_m\geq x_0\) for all sufficiently large \(m\). Applying the
preceding argument and then letting \(m\to\infty\) proves the
hard-region estimate.

Lemma~\ref{lem:easy} handles the complementary region
\(P/F\leq1+n/(2k^2)\). Thus, under the normalization
\(a=\lambda_1=1\), we have
\(\mathcal Q_\gamma(\lambda;\xi)\geq0\) whenever
\(0<F\leq\eta_*\).

Finally, for a general scale, set
\(\widehat\lambda:=\lambda/a\) and
\(\widehat\xi:=\xi/a\). Then
\(\sigma_k(\widehat\lambda)=\sigma_k(\lambda)/a^k\), while
homogeneity gives
\[
    \mathcal Q_\gamma(\lambda;\xi)
    =
    a^{k-1}
    \mathcal Q_\gamma(\widehat\lambda;\widehat\xi).
\]
The normalized result therefore proves
\eqref{Qbetagamma>0}.
\end{proof}

\section{Applications to Curvature and Hessian Estimates}\label{applications}
In this section, we apply Theorem~\ref{thm-crucial-ineq} to derive global curvature
estimates for prescribed curvature equations and the corresponding
global-to-boundary estimates for Hessian equations.

\begin{proof}[Proof of Theorem~\ref{thm:curvature-application}.]

Let
\begin{eqnarray*}
    u:=\langle X,\nu\rangle
\end{eqnarray*}
be the support function. Since $M$ is star-shaped, we have $C_0>u>c_0>0$.

Consider the test function
\begin{eqnarray*}
    G:=\log\kappa_{\max}-N\log u,
\end{eqnarray*}
where the large constant \(N>1\) will be chosen later. Suppose that \(G\) attains its
maximum at \(X_0\). Choose an orthonormal frame at \(X_0\) such that \(h=(h_{ij})\) is
diagonal, and write
\begin{eqnarray*}
    a:=\kappa_1=\cdots=\kappa_m
    >\kappa_{m+1}\geq\cdots\geq\kappa_n.
\end{eqnarray*}
When \(m>1\), as in Remark \ref{rmk-multiple}, we have $ \xi_p=h_{pp;1}=0$ for $2\leq p\leq m$. Thus Corollary~\ref{cor:multiple-largest-eigenvalues} applies. For simplicity, the
following calculation is written for \(m=1\); when \(m>1\), the sum
over \(p\) begins at \(p=m+1\). If $a$ is bounded, there is nothing to prove. We may therefore assume that $a>1$ is sufficiently large.

At \(X_0\), we have
\begin{eqnarray}\label{eq:curvature-critical-point}
    0
    =
    G_i
    =
    \frac{h_{11;i}}{a}
    -N\frac{u_i}{u}.
\end{eqnarray}
The second derivative formula for the largest eigenvalue gives
\begin{eqnarray}
    0
    &\geq&
    \sum_{i=1}^nF^{ii}G_{ii}
    \nonumber\\
    &\geq&
    \frac{1}{a}\sum_{i=1}^nF^{ii}h_{11;ii}
    +\frac{2}{a}
    \sum_{i=1}^n\sum_{p=2}^n
    \frac{F^{ii}h_{1p;i}^2}{a-\kappa_p}
    -\frac{1}{a^2}
    \sum_{i=1}^nF^{ii}h_{11;i}^2
    \nonumber\\
    &&
    -\frac{N}{u}\sum_{i=1}^nF^{ii}u_{ii}
    +\frac{N}{u^2}\sum_{i=1}^nF^{ii}u_i^2.
    \label{eq:curvature-second-maximum}
\end{eqnarray}

We next differentiate the curvature equation. Its first derivative is
\begin{eqnarray}\label{eq:curvature-equation-first}
    \sum_{i=1}^nF^{ii}h_{ii;l}
    =
    d_Xf(e_l)+\kappa_l d_\nu f(e_l),
\end{eqnarray}
while a second differentiation in the \(e_1\)-direction gives
\begin{eqnarray*}
    \sum_{i=1}^nF^{ii}h_{ii;11}
    +
    \sum_{p,q,r,s=1}^nF^{pq,rs}h_{pq;1}h_{rs;1}
    \geq
    \sum_{l=1}^n h_{11;l}d_\nu f(e_l)
    -C(1+a^2),
\end{eqnarray*}
and therefore
\begin{eqnarray}\label{eq:curvature-equation-second}
\sum_{i=1}^nF^{ii}h_{ii;11}\geq -\sum_{i,j=1}^nF^{ii,jj}h_{ii;1}h_{jj;1}+2\sum_{p=2}^n\frac{F^{pp}-F^{11}}{a-\kappa_p}h_{11;p}^2+\sum_{l=1}^n h_{11;l}d_\nu f(e_l)
    -C(1+a^2).   
\end{eqnarray}
By \eqref{eq:curvature-critical-point},
\eqref{eq:curvature-equation-first}, and
\(u_l=\kappa_l\langle X,e_l\rangle\), the terms containing
\(d_\nu f\) cancel:
\begin{eqnarray*}
    \frac{1}{a}
    \sum_{l=1}^n h_{11;l}d_\nu f(e_l)
    -
    \frac{N}{u}
    \sum_{l=1}^n
    \langle X,e_l\rangle
    \sum_{i=1}^nF^{ii}h_{ii;l}=
    -\frac{N}{u}
    \sum_{l=1}^n
    d_Xf(e_l)\langle X,e_l\rangle
    \geq-CN.
\end{eqnarray*}

Substituting these identities into
\eqref{eq:curvature-second-maximum}, using
\begin{eqnarray*}
    h_{11;ii}
    =
    h_{ii;11}
    +a\kappa_i(a-\kappa_i),
\end{eqnarray*}
and collecting the third-order terms, we obtain
\begin{eqnarray}\label{eq:curvature-key-inequality}
    0
    &\geq&
    \mathcal T_\gamma+(\gamma-1)\frac{F^{11}}{a^2}h_{11;1}^2+\sum_{p=2}^n\frac{a+\kappa_p}{a^2(a-\kappa_p)}F^{pp}h_{11;p}^2\nonumber\\
    &&+(N-1)\sum_{i=1}^nF^{ii}\kappa_i^2
    +\frac{N}{u^2}\sum_{i=1}^nF^{ii}u_i^2
    -C(1+N+a),
\end{eqnarray}
where
\begin{eqnarray}\label{eq:curvature-third-order-block}
    \mathcal T_{\gamma}
    :=
    -\frac{1}{a}\sum_{i,j=1}^nF^{ii,jj}\xi_i\xi_j
    -\gamma\frac{F^{11}}{a^2}\xi_1^2
    +\frac{2}{a}\sum_{p=2}^n
    \frac{F^{pp}}{a-\kappa_p}\xi_p^2,\qquad \xi_i:=h_{ii;1}.
\end{eqnarray}
The constant \(\gamma\in(0,1]\) will be chosen below. We have
also discarded the nonnegative terms involving \(h_{ij;k}^2\) for
which \(i,j,k\) are pairwise distinct.

For the terms involving \(h_{11;p}^2\), \eqref{n-k/k} and \(k\geq n/2\) give \(a+\kappa_p\geq0\); hence these terms may also be discarded.

We next estimate \(\mathcal T_\gamma\). By the definition of
\(\mathcal Q_\gamma\),
\begin{eqnarray}\label{eq:T-and-Q1}
    \mathcal T_\gamma
    =
    \mathcal Q_\gamma(\kappa;\xi)
    -
    \frac{2}{aF}
    \left(
        \sum_{i=1}^nF^{ii}\xi_i
    \right)^2.
\end{eqnarray}
Taking \(l=1\) in \eqref{eq:curvature-equation-first}, we have
\begin{eqnarray*}
    \sum_{i=1}^nF^{ii}\xi_i
    =
    d_Xf(e_1)+a\,d_\nu f(e_1).
\end{eqnarray*}
It follows from Theorem \ref{thm-crucial-ineq} that
\begin{eqnarray}\label{eq:curvature-equation-error}
    \mathcal T_\gamma
    \geq
    -C(1+a).
\end{eqnarray}

When \(k>n/2\), choose \(\gamma=1\). It follows that the first line of
\eqref{eq:curvature-key-inequality} is bounded below by
\(-C(1+a)\).

When \(k=n/2\), fix \(N\) and choose
\begin{eqnarray*}
    \gamma:=1-\varepsilon,
    \qquad
    0<\varepsilon<\frac{1}{2N}.
\end{eqnarray*}
By the critical equation \eqref{eq:curvature-critical-point}, we have
\begin{eqnarray*}
    -\varepsilon\frac{F^{11}}{a^2}h_{11;1}^2
    +\frac{N}{u^2}\sum_{i=1}^nF^{ii}u_i^2
    \geq
    \left(
        N-\varepsilon N^2
    \right)
    \frac{F^{11}u_1^2}{u^2}
    \geq0.
\end{eqnarray*}

Thus, in either case, \eqref{eq:curvature-key-inequality} gives
\begin{eqnarray}\label{eq:curvature-final-estimate}
    0
    \geq
    (N-1)\sum_{i=1}^nF^{ii}\kappa_i^2
    -C(1+N+a).
\end{eqnarray}
Finally, the Newton--Maclaurin inequalities imply
\begin{eqnarray*}
    \sum_{i=1}^nF^{ii}\kappa_i^2
    \geq
    ca.
\end{eqnarray*}
Choosing \(N\) sufficiently large, we deduce that 
\begin{eqnarray*}
    a\leq C.
\end{eqnarray*}
\end{proof}

We also have the following global-to-boundary estimates for $k$ Hessian equations:
\begin{theorem}\label{thm:euclidean-hessian-application}
Let \(n\geq3\) and \(n/2\leq k<n\). Suppose that
\(\Omega\subset\mathbb R^n\) is a bounded domain and that
\(u\in C^4(\Omega)\cap C^2(\overline\Omega)\) is a
\(k\)-admissible solution of
\begin{eqnarray}\label{eq:euclidean-hessian-application}
    \sigma_k(D^2u)
    =
    f(x,u,Du)>0,
    \qquad x\in\Omega,
\end{eqnarray}
for some positive function
\(f(x,z,p)\in C^2\). Then there exists
a constant \(C\), depending only on
\(n\), \(k\), \(\Omega\), \(\|u\|_{C^1(\overline\Omega)}\),
\(\inf f\), and \(\|f\|_{C^2}\), such that
\begin{eqnarray}\label{eq:euclidean-global-C2-application}
    \max_{x\in\overline\Omega}|D^2u(x)|
    \leq
    C\left(
        1+\max_{x\in\partial\Omega}|D^2u(x)|
    \right).
\end{eqnarray}
\end{theorem}

\begin{proof}
Consider
\begin{eqnarray*}
    G
    :=
    \log\lambda_{\max}(D^2u)
    +\frac{A}{2}|Du|^2.
\end{eqnarray*}
The proof is very similar to Theorem~\ref{thm:curvature-application}, so we omit it.
\end{proof}

\section*{Acknowledgments}
The author would like to thank Professor Guohuan Qiu for helpful
discussions on this problem and Professor Xi-Nan Ma for his
longstanding support. The author acknowledges support from Grant 2025YFA1017603 of the National Key R\&D Program of China.

\section*{AI Statement}

During the preparation of this manuscript, the author used ChatGPT
(OpenAI) to assist with language editing and the presentation of the
material. The author carefully reviewed and independently verified all
AI-assisted output and takes full responsibility for the content of
the manuscript.

\printbibliography

\end{document}